\documentclass{myart2}
\begin{document}
	
	\title{Maximal attractors for perturbations of unimodal maps near a homoclinic tangency}
	\author{Przemys{\l}aw Kucharski\footnote{
			This work was supported by the National Science Centre, Poland (NCN), ~grant no. 2019/3
            4/E/ST1/00237. \\
			ORCiD: 0000-0002-3826-5827,\\
			\emph{email:} przemyslaw.kucharski@doctoral.uj.edu.pl.
	}}
	\affil{Jagiellonian University, Faculty of Mathematics and Computer Science, 
		Prof. St. {\L}ojasiewicza St 6, PL30348, Cracow, Poland}
	\maketitle
	\begin{abstract}
		We show that small planar perturbations of families of unimodal maps have maximal (equal to the intersection of the iterates of their trapping region) attractors at parameters near homoclinic tangency of one of the fixed points. The results hold for general unimodal maps exhibiting certain hyperbolic properties. Hence, admissible families considered include H\'enon-like families, Lozi-like families, or perturbations of border collision normal forms. Therefore, our results generalize results of Viana \cite{viana:global-att} in the case of H\'enon-like families and Glendinning and Simpson \cite{simpson:robust-chaos} in the case of border collision normal forms. 
        
	\end{abstract}

	\section*{Outline of the article}
		\subsection*{Brief summary}
			In this work, we consider planar perturbations of families of unimodal maps, called admissible families, parameterized in such a way as to model the unfolding of a homoclinic tangency. The involved unimodal maps are assumed to exhibit certain hyperbolic properties, and therefore admissible families considered include H\'enon-like families, Lozi-like families, or perturbations of border collision normal forms.
			
			The key component of the proofs is the renormalization model introduced by Dyi-Shing Ou in \cite{dyi-shing-ou:critical-points-I} for the orientation preserving Lozi family. In this work, we generalize this model to admissible families. 
			
			We show that sufficiently small perturbations of admissible families have maximal (equal to the intersection of the iterations of their trapping region) attractors at parameters near homoclinic tangency of one of the fixed points; see Theorem \ref{mainthm:maximal-attractors}. Our results are novel and generalize results of Viana \cite{viana:global-att} in the case of H\'enon-like families, Cao and Liu \cite{cao-liu-98}, and the author \cite{kucharski:str-att}, in the case of the Lozi family, and Glendinning and Simpson \cite{simpson:robust-chaos} in the case of border collision normal forms. 
			
			In the case of piecewise hyperbolic admissible families, we show that the closure of the unstable manifold of one of the fixed points is an invariant set, on which the system is transitive and has Devaney chaos, see Theorem \ref{mainthm:chaotic-attractors}. This generalizes multiple works, in particular the results on Lozi-like maps \cite{Lozi-likemaps}, border-collision normal forms \cite{simpson:robust-chaos}.
			
		\subsection*{Structure and results}
		We outline the structure of the work. In generality presented in this work, all results are novel.
		
		Section \ref{section:admissible-families} introduces admissible families. The definition is based on the Wang-Young families \cite{StrangeAttractorswithOneDirectionofInstability}, but is more general, and encompasses not only non-uniformly hyperbolic families, but also piecewise hyperbolic families such as those in \cite{Young1985BowenRuelleMF}, and in particular Lozi-like families or the family of border collision normal forms. A very concise sketch of the proof that admissible families exhibit certain forms of hyperbolicity, see Lemma \ref{lemma:cone-inv}, can be found in \cite{StrangeAttractorswithOneDirectionofInstability}. However, the details of the proof are not present in \cite{StrangeAttractorswithOneDirectionofInstability}, perhaps because the reader is assumed to already be familiar with these techniques. 
		
		Section \ref{section:piecewise-uni-hyp-families} contains the proof that, under the assumption of piecewise uniform hyperbolicity and high expansion, any map of the admissible family is mixing on the closure of the unstable manifold $\cl\UnstableMan{z_{+}}$ of the fixed point $z_{+}$; see Theorem \ref{mainthm:chaotic-attractors}. The proof follows an approach first used in \cite{strange-attractor-mis}. The results of Theorem \ref{mainthm:chaotic-attractors} are novel in this setting and generalize the results of \cite{simpson:robust-chaos,Lozi-likemaps, strange-attractor-mis}. The fact that the family of border collision normal forms is an admissible family is presented in Section \ref{subsection:c1-perturbations}.
		
		The renormalization model is developed in Section \ref{section:renormalization}. The model is based on the results of Ou \cite{dyi-shing-ou:critical-points-I} and constitutes their substantial and non-trivial generalization with quantitative differences in the behavior of the induced map of the model; see Proposition \ref{proposition:multi-renorm} and the discussion in Section \ref{section:maximal-attractors}. The model is crucial to the proof of maximality of attractors, see Theorem \ref{mainthm:maximal-attractors}, and the fact that there are trapping regions containing invariant set outside the chaotic attractor of Theorem \ref{mainthm:chaotic-attractors}.
		
		Section \ref{section:maximal-attractors} contains the proof of Theorem \ref{mainthm:maximal-attractors} and preliminary lemmas. Theorem \ref{mainthm:maximal-attractors} generalizes the results of Viana \cite{viana:global-att} in the case of orientation preserving H\'enon-like maps, the author's \cite{kucharski:str-att}, Cao and Liu \cite{cao-liu-98} for the Lozi family, and Glendinning and Simpson for the orientation preserving family of border collision normal forms \cite{simpson:robust-chaos}. The proof of Theorem \ref{mainthm:maximal-attractors} is based on a novel application of the renormalization model compared to Ou results \cite{dyi-shing-ou:critical-points-I}. 
		
		In Section \ref{subsection:hasudorff-cont} we show that the attractors vary continuously in the Hausdorff metric, as long as they are the maximal invariant subset of some trapping region varying continuously, see Theorem \ref{mainthm:maximal-attractors} and Proposition \ref{prop:continuity-hausdorff}. The results and methods of this section are novel.
    \tableofcontents
		
	\section{Introduction}
	
	\subsection*{Asymptotic behaviour and hyperbolicity}One of the main themes in the theory of dynamical systems is the study of asymptotic behavior of orbits. That is, one considers a map $f\colon X\to X$, where $X$ is a topological space (often with additional structures such as a measure or a structure of a manifold), and $f$ is a continuous or even smooth endomorphism of $X$. Then various properties of the orbits $\{z,f(z),f^{2}(z),\ldots\}$, for $z\in X$, are studied. Depending on the existing structures in the system, these properties can be of metric, topological, or geometric origin. In short, dynamical system theorists are often concerned with the description of the limiting behavior of points in terms of existing structures.
	
	In the particularly relevant case of smooth dynamical systems, we can distinguish a class of systems with so called hyperbolic properties. Hyperbolicity along an orbit of a point $z\in X$, for some manifold $X$, in loose terms, is exponential divergence or convergence of nearby points to $z$. Classical theory ensures in this case the existence of stable and unstable manifolds, which, locally to $z$, are smooth manifolds with tangent spaces spanning the tangent space $T_{z}X$, see Section \ref{section:definitions} for precise definitions. In a way, relative positions of those manifolds, their geometrical properties, the way they intersect each other, govern some portion of the dynamics of the system. 
	
	\subsection*{Unfolding of a homoclinic tangency}One of the scenarios often studied is that of the unfolding of a homoclinic tangency. Specifically, a parameterized family $\{f_{t}\colon X\to X\}_{t\in T}$ of smooth maps is considered, in which at certain parameter $t_{0}\in T$ there exists a fixed point $z_{0}$ in the system $(f_{t_{0}},X)$ and its stable manifold has a tangential intersection with its unstable manifold. This situation is of particular relevance, since one can prove that near the parameter $t_{0}$ there is a very rich dynamical behavior \cite{palistakens:hyp}, such as, for example, the existence of horseshoes or attractors with chaotic properties \cite{viana:global-att,viana-marcelo-strange-att-1993}. Moreover, unfolding of a homoclinic tangency can be modeled by embedding a one dimensional unimodal map into the plane and perturbing it to obtain a diffeomorphism, or diffeomorphism outside a singularity set. Therefore, perturbations of unimodal maps with some hyperbolic properties (which are then inherited by the perturbed system), such as in \cite{viana:global-att,viana-marcelo-strange-att-1993,StrangeAttractorswithOneDirectionofInstability}, or some specific perturbed family that models a generic unfolding of a tangency, such as H\'enon \cite{henon1976}, or border-collision normal forms \cite{simpson:robust-chaos,simpson-gosh:robust-chaos}. The strategies are similar both in the case of smooth maps, which have non-uniformly hyperbolic behavior, and in the case of piecewise smooth maps, which have hyperbolic behavior along orbits at which the map is differentiable. We will often differentiate between the smooth, non-uniformly hyperbolic case (such as H\'enon-like maps, or Wang-Young families \cite{henon-dynamics-of,StrangeAttractorswithOneDirectionofInstability, viana-marcelo-strange-att-1993}), and the piecewise smooth, piecewise hyperbolic case, such as border collision normal forms, Lozi-like maps and their small perturbations \cite{simpson:robust-chaos,Lozi-likemaps,strange-attractor-mis}. 
	
	
	\subsection*{Attractors and chaotic properties}The existence of strange, that is, with chaotic properties (see Section \ref{section:definitions} for definitions), attractors for the unfolding of a generic quadratic homoclinic tangency has been established incrementally in a series of works \cite{henon-dynamics-of, mora1993,ures_1995,viana-marcelo-strange-att-1993, viana:global-att}. On the other hand, in the case of piecewise affine border collision normal forms, abbreviated BCNFs, (which can be considered a model for piecewise affine and piecewise hyperbolic maps, since the BCNFs family includes all piecewise affine planar maps with two pieces of affinity), the existence of strange or chaotic attractors has been established first by Misiurewicz \cite{strange-attractor-mis} in the case of a subfamily of orientation reversing Lozi maps, and then by other authors, in particular by Simpson and Glendinning for orientation preserving BCNFs. All mentioned results were restricted to some parameter regions where the system has hyperbolic properties. In the case of, arguably simpler, piecewise hyperbolic BCNFs, the region for the existence of chaotic attractors can be computed explicitly, while in the case of unfolding of a quadratic generic tangency, it is necessary to assume that the system is close to one dimension, that is, the system is a small perturbation of a one dimensional system. This allows us to deduce some properties of the perturbed system which are rooted in the analogous properties of the one dimensional map. Our results also use the perturbative argument. Hence, if our results are applied to the BCNFs, we do not obtain a full generalization of the Glendinning and Simpson results. Specifically, we obtain a smaller region where the chaotic attractor exists, that is, the attractor with sensitive dependence on initial condition. On the other hand, we obtain regions where a strange attractor exists, a result not present in Glendinning and Simpson's work. A strange attractor is an attractor with an orbit with positive Lyapunov exponent that is maximal, that is, the attractor is equal to the intersection of the iterates of its trapping region. The latter is not required in the definition of a chaotic attractor. Hence, a strange attractor is in a way stronger a notion than a chaotic attractor. 
	
	\subsection*{Maximality of attractors}Apart from the existence of attractors in admissible families, we investigate whether the strange or chaotic attractors constitute, together with the fixed points, the non-wandering set. A non-wandering set is the set where all non trivial dynamics occurs in the system. Hence, the structure of non-wandering set is crucial in understanding the system. 
	Let $f\colon\plane\to\plane$ be a member of an admissible family. By the nature of admissible families, strange or chaotic attractors always contain the closure of the unstable manifold of a hyperbolic fixed point $z_{+}$ of $f$, denoted $\cl\UnstableMan{z_{+}}$, and in all known examples to the author, are, in fact, equal to $\cl\UnstableMan{z_{+}}$. In the case of orientation reversing admissible families, it is straightforward to prove that the non-wandering set is equal to $\cl\UnstableMan{z_{+}}\cup\{z_{-}\}$, where $z_{-}$ and $z_{+}$ are hyperbolic fixed points. On the other hand, in the orientation preserving case this problem is solved only partially, and only for some families. It is known that in the case of H\'enon-like maps \cite{viana:global-att,cao-mao:non-wandering-set-of-some-Henon-maps}, there are open parameter regions accumulating on the parameter of homoclinic tangency, in which the strange attractor $\cl\UnstableMan{z_{+}}$ is maximal (that is equal to $\bigcap f^{n}(U)$ for some trapping region $U\supset \cl\UnstableMan{z_{+}}$), equivalently in this case, global. Similarly, as a piecewise hyperbolic family, in the case of Lozi maps \cite{kucharski:str-att,cao-liu-98}, it is known that similar phenomena occur. Nevertheless, it is not known whether there are parameter regions such that $\cl\UnstableMan{z_{+}}$ is a chaotic or strange attractor, and there are non-wandering points outside $\cl\UnstableMan{z_{+}}$. An affirmative answer to this problem in the case of H\'enon-like maps would be significant, since it would imply a possible direction in the study of coexistence phenomena in  non-uniformly hyperbolic families, which is a part of the famous Palis program.

	\begin{thmmain}\label{mainthm:chaotic-attractors}
		Le $\cc{F}:=\{f_{a,b}\}_{(a,b)\in A\times B}$ be a piecewise hyperbolic admissible family, satisfying \ref{stepA2-expansion} and \ref{stepA2-hyp}, and $a_{*}$ be the parameter such that $f_{a_{*}}(c)=q_{+}$, and $f_{a}(c)>q_{+}$ for $a\in (a_{*},a_{2})$. Then there exists $\tilde b\in B$ such that for $b\in (0,\tilde b)$, there are ascending intervals $\tilde A^{b}\subset A$ with $\bigcup_{b\in (0,\tilde b)}\tilde{A}^{b} =(\tilde a,a_{*})$, for some $\tilde a\in A$, and such that $\cl\UnstableMan{z_{+}}$ is a chaotic attractor in $\tilde{A}\times\{b\}$. 
	\end{thmmain}
	\begin{thmmain}\label{mainthm:maximal-attractors}
		Let $\cc{F}:=\{f_{a,b}\}_{(a,b)\in A\times B}$ be an admissible family and $a_{*}$ be the parameter such that $f_{a_{*}}(c)=q_{+}$, and $f_{a}(c)>q_{+}$ for $a\in (a_{*},a_{2})$. Then there are open sets $K_{n}\subset A$, with $K_{n}\to \{a_{*}\}$, as $n\to\infty$, in the Hausdorff metric, and parameters $k_{n}\in B$, with $k_{n}\to 0$, as $n\to\infty$, such that the $\cl\UnstableMan{z_{+}}$ is a maximal invariant set in some neighbourhood for parameters in $K_{n}\times (0,k_{n})$. In particular, at those parameters, under the additional assumption of piecewise uniform hyperbolicity \ref{stepA2-hyp} and high expansion \ref{stepA2-expansion}, $\cl\UnstableMan{z_{+}}$ is a topological strange attractor and varies continuously in the Hausdorff metric. 
	\end{thmmain}

	\section{Basic definitions}\label{section:definitions}
	\subsection{Notation}We adhere to the following notation \begin{enumerate}
		\item For $\delta>0$, define $B_{\delta}(z)\subset X$ to be the ball of radius $\delta$ and center $z\in X$.
		\item For a subset $U\subset X$, of a topological space $X$, let $\cl U$, $\inter U$, $\fr U$ be, respectively, the closure, the interior, and the boundary of $U$.
		\item For a diffeomorphism $f\colon M\to M$ on a manifold $M$, denote by $df\colon TM \to TM$ its derivative, where $TM=\bigcup_{z\in M} T_{z}M$ is the tangent space. By $df_{z}$, we denote the derivative at $z\in M$.
		\item For a piecewise smooth curve $\gamma$, let $\length(\gamma)$ be its length.
		\item For a subset $U\subset X$, let $\diam U$ be the diameter of $U$.
		\item For $z\in X$ and $U\subset X$, let $d(z, U)$ be the distance from $z$ to $U$.
	\end{enumerate}
	
	\subsection{Attractors}
	Let $h\colon X\to X$ be a continuous transformation of a compact metric space. The \emph{$\omega$-limit set} $\omega(x)$ of a point $x$ is defined as \[\omega_{h}(x)=\bigcap_{n\in\N}\cl\{h^{k}(x)\colon k\geq n\}.\] An open subset $U\subset X$ will be called a \emph{trapping region} if $h(\cl U)\subset U$. A compact set $F\subset X$ contained in a trapping region $U$ and satisfying $\bigcap_{n\in\N}h^{n}(U)=F$ will be called \emph{maximal} in $U$. For an $h$-invariant $F$, we define \emph{the stable set} $W^{s}(F)$ of $F$ as all those points $x\in X$ such that $\omega(x)\subset F$. A map $h$ is said to be \emph{transitive} if for any open sets $W,V\subset X$ we can find $i\in\N$ such that $h^{i}(W)\cap V\neq\emptyset$. For compact metric spaces without isolated points, transitivity is equivalent to the existence of a point with dense orbit. If $h^{i}(W)\cap V\neq\emptyset$ holds for almost every $i\in\N$, then $h$ is said to be \emph{mixing}. For the case of $h\colon \plane\to\plane$ differentiable along the orbit of some $x\in \plane$, let us also recall that for every $v\in\plane$ we can define \emph{the lower Lyapunov exponent} at $(x,v)\in T_{x}\plane$ by $\chi_{-}(x,v):=\liminf_{n\to\infty}\frac{1}{n}\log||df_{x}^{n}v||$, where $df_{y}$ is the derivative of $f$ at $y\in \plane$.
	\begin{definition}\label{def:attractor}
		A subset $A\subset X$ is said to be an \emph{attractor} if it is a compact $h$-invariant set with a dense orbit and its stable set has non-empty interior. If additionally $A$ is maximal in some trapping region, it is called a \emph{topological attractor}.
	\end{definition}
	Recall that $h\colon X\to X$ has \emph{the sensitive dependence on initial conditions} if there is $\delta>0$ such that for every $x\in X$ and neighborhood $N$ of $x$ there exist a $y\in N$ and an $n\geq 0$ with $d(h^{n}(x),h^{n}(y))>\delta$.
	\begin{definition}\label{def:chaotic-att}
		If $A$ is an attractor, $h|_{A}$ is transitive, the periodic points are dense in $A$ and $h|_{A}$ has sensitive dependence on initial conditions, then $A$ will be called a \emph{chaotic attractor} and is said to have \emph{Devaney chaos}.
	\end{definition}
	Note that the above definition excludes periodic attractors. Although they are transitive and have a dense set of periodic orbits, they do not verify the definition of sensitive dependence on initial conditions. Moreover, a finite set can satisfy Definition \ref{def:attractor} only if it is comprised of a periodic source.

	\begin{definition}\label{def:strange-attractor}
		A subset $A\subset X$ is said to be a \emph{topological strange attractor} if it is an attractor, has a point with a positive lower Lyapunov exponent along a dense orbit, and is a maximal set of some trapping region.
	\end{definition}
	In the literature, we can find many non-equivalent definitions of a strange attractor. The choice of a definition is dictated by intuitions regarding the matter of strangeness, that is, in which exactly area the strangeness is exhibited; geometric, topological, dynamical, or maybe metric. Our definitions are formulated so that the concept of strangeness is somewhat compatible with the relevant works to our results, such as \cite{Lozi-likemaps,strange-attractor-mis}.
	\subsection{Hyperbolic behavior}
	\subsubsection*{Families of cones}
	\begin{definition}
		For a number $\rho\in(0,1)$, called the coefficient of cone, we define a \emph{stable $\rho$-cone} $C^{s}_{\rho}\subset\plane$ by \begin{equation}\label{eq:cone-dfn-stable}
			C^{s}_{\rho}:=\{(v_{1},v_{2})\in\plane\colon|v_{1}|\leq \rho |v_{2}|\},
		\end{equation}and an \emph{unstable $\rho$-cone} $C^{u}_{\rho}\subset\plane$ by \begin{equation}\label{eq:cone-dfn-unstable}
			C^{u}_{\rho}:=\{(v_{1},v_{2})\in\plane\colon|v_{2}|\leq \rho |v_{1}|\}.
		\end{equation}The interiors $\inter C^{s}_{\rho}$ and $\inter C^{u}_{\rho}$ are defined as\begin{align*}
		   \inter C^{s}_{\rho}&:=\{(v_{1},v_{2})\in\plane\colon|v_{1}|< \rho |v_{2}|\}\cup\{(0,0)\}\text{ and }\\
           \inter C^{u}_{\rho}&:=\{(v_{1},v_{2})\in\plane\colon|v_{2}|< \rho |v_{1}|\}\cup\{(0,0)\}. 
		\end{align*} 
	\end{definition}
	\begin{definition}[A family of cones]
		Let $\rho>0$. The families $\{C^{\kappa}_{z,\rho}\subset T_{z}\plane\}_{z\in\plane}$, for $\kappa\in\{s,u\}$, will be called,  respectively, \emph{stable} and \emph{unstable families of $\rho$-cones}, or \emph{vertical} and \emph{horizontal} families of $\rho$-cones.
    \end{definition}
	For our results, we will only need families of cones with constant coefficients, and so we restrict the definitions to such a case. The following definitions depend on the choice of the families of cones. The choice will be made later, as we introduce the primary object of this work, called admissible families of maps. For now, we will drop the dependence on $\rho>0$.
	
	\subsubsection*{Stable and unstable curves}
	For a piecewise smooth curve $\gamma\colon I:=[0,1]\to \plane$ let $\gamma'=(\gamma_{x}',\gamma_{y}')$ be the section of vectors tangent to $\gamma$ restricted to its domain of differentiability. We define the length of $\gamma$ by $\length(\gamma):=\int ||\gamma'||$.
	\begin{definition}
The $C^{1}$-smooth curves $\gamma_{s}$, $\gamma_{u}$ with tangent vectors contained in, respectively, stable and unstable $\rho$-cone families, that is, $\gamma'_{s}(t)\in C^{s}_{\gamma_{s}(t),\rho}$ and $\gamma'_{u}(t)\in C^{u}_{\gamma_{u}(t),\rho}$, for $t\in I$, will be called a \emph{$s$-curve} and a \emph{$u$-curve}. 
	\end{definition}
	
	\begin{definition}\label{definition:rectangle}
		We will say that a closed disc $D\subset \plane$ is a \emph{rectangle} if its boundary $\fr D$ consists of four smooth curves, two of which are disjoint $s$-curves, and two are disjoint $u$-curves. The $u$-curves will be called horizontal faces and the $s$-curves vertical faces.
	\end{definition}
	The following lemma will be useful later. It can be found in \cite{Lozi-likemaps}.
	\begin{lemma}\label{lemma:cone->gamma-a-function}
		Any $u$-curve can be expressed as $\gamma_{u}=\{(x,\phi_{u}(x))\colon x\in I_{u}\}$ for some $C^{1}$ function $\phi_{u}\colon I_{u}\to \R$, where $I_u=\R$ or $I_u=[0,1]$. Similarly, any $s$-curve can be expressed as $\gamma_{s}=\{(\phi_{s}(y),y)\colon y\in I_{s}\}$ for some $C^{1}$ function $\phi_{s}\colon I_{s}\to \R$, where $I_s=\R$ or $I_s=[0,1]$.
	\end{lemma}
	
	\begin{definition}\label{definition:strip}
		For a rectangle $U$, we will say that a curve \emph{$\gamma$ is on $U$}, if $\gamma$ disconnects $U$ into exactly two connected components, and connects the $u$-curves of $\fr U$. The \emph{interior} of $\gamma$ is defined as $\gamma$ without its endpoints. For $\gamma_{1}$, $\gamma_{2}$ curves on $U$, which have disjoint interiors inside $U$, we define the \emph{strip $S(U,\gamma_{1},\gamma_{2})$} as the closure of the connected component of $U\setminus\{\gamma_{1}\cup\gamma_{2}\}$ that has $\gamma_{1}$ and $\gamma_{2}$ in the boundary. If no ambiguity can arise, we will shorten the notation to $S(\gamma_{1},\gamma_{2})$.
	\end{definition}
	\begin{definition}
		Let $U,U'$ be rectangles. If there exist curves $\gamma_{1}$ and $\gamma_{2}$ such that $U'=S(U,\gamma_{1},\gamma_{2})$, then $U'$ will be said to be a \emph{subrectangle} of $U$.
	\end{definition}
	
	\subsubsection*{Hyperbolic sets}
	In this section, we outline basic concepts about hyperbolic sets. The section is based on Chapter 6 of \cite{katok_Hasselblatt_1995}. 
	
	\begin{definition}
		Let $\{L_{m}\colon \plane\to\plane\}_{m\in\Z}$ be a sequence of invertible linear maps and $\lambda<\mu$. We say that $\{L_{m}\}_{m\in\Z}$ \emph{admits a $(\lambda,\mu)$-splitting} if there exist decompositions $\plane=E^{+}_{m}\oplus E^{-}_{m}$ such that $L_{m}E^{\pm}_{m}=E_{m+1}^{\pm}$, and \[
		\norm{L_{m}|_{E_{m}^{-}}}\leq \lambda, \quad \norm{L^{-1}_{m}|_{E^{+}_{m+1}}}\leq \mu^{-1}.
		\]We will call  $\{L_{m}\}_{m\in\Z}$ \emph{hyperbolic} if it admits a $(\lambda,\mu)$-splitting for some $\lambda<1<\mu$.
	\end{definition}
	Let us assume $M$ is a $C^{1}$ smooth surface, $U\subset M$ is an open subset,  and $f\colon U\to M$ a $C^{1}$ diffeomorphism onto its image and $\Lambda\subset U$ a compact $f$-invariant set.
	\begin{definition}
		The set $\Lambda$ is called a \emph{hyperbolic set} for the map $f$ if there exists a Riemannian metric in an open neighborhood $U$ of $\Lambda$ and $\lambda<1<\mu$ such that for any point $x\in\Lambda$ the sequence of differentials $df_{f^{n}(x)}\colon T_{f^{n}(x)}M\to T_{f^{n+1}(x)}M$, $n\in\Z$, admits a $(\lambda,\mu)$-splitting.
	\end{definition}
	
	Let us recall a useful cone criterion for hyperbolic sets, see Corollary 6.4.8 in \cite{katok_Hasselblatt_1995}.
	\begin{proposition}\label{proposition:cone-criterion-hyp-set}
		A compact $f$-invariant set $\Lambda$ is hyperbolic if there exist $\lambda<1<\mu$ such that for every $x\in\Lambda$ there is a decomposition $T_{x}M=S_{x}\oplus T_{x}$, a family of horizontal cones $C^{u}_{\rho,x}\supset S_{x}$, and a family of vertical cones $C^{s}_{\rho,x}\supset T_{x}$, for some $\rho>0$, such that \begin{align}\label{proposition:cone-criterion-hyp-set:eq:cone}
			df_{x}C^{u}_{\rho,x}\subset \inter C^{u}_{\rho,f(x)},&\quad df^{-1}_{f(x)}C^{s}_{\rho,f(x)}\subset \inter C^{u}_{\rho,x},\\
			\norm{df_{x}v}\geq \mu\norm{v},\text{ for }v\in C^{u}_{\rho,x},&\text{ and }\norm{df^{-1}_{f(x)}w}\geq \lambda^{-1}\norm{w},~\text{ for }w\in C^{s}_{\rho,f(x)}.
		\end{align}
	\end{proposition}
	
	We are now ready to state the crucial result of this section; see Theorem 6.4.9 in \cite{katok_Hasselblatt_1995}.
	\begin{theorem}\label{theorem:inv-manifolds-hyp-set}
		Let $\Lambda$ be a hyperbolic set for a $C^{1}$ diffeomorphism $f\colon U\to M$ such that $df$ admits a $(\lambda,\mu)$-splitting on $\Lambda$ with $\lambda<1<\mu$. Then for each $x\in\Lambda$ there is a pair of embedded $C^{1}$ discs $\StableLocMan{x}$ and $\UnstableLocMan{x}$, called the local stable manifold and the local unstable manifold of $x$, respectively, such that \begin{enumerate}[label=(\roman*)]
			\item\label{theorem:inv-manifolds-hyp-set:tangents} $T_{x}\StableLocMan{x}=E^{-}_{x}$, $T_{x}\UnstableLocMan{x}=E^{+}_{x}$,
			\item $f(\StableLocMan{x})\subset \StableLocMan{f(x)}$, $f^{-1}(\UnstableLocMan{x})\subset \UnstableLocMan{f^{-1}(x)}$,
			\item for every $\delta>0$ there exists $C(\delta)$ such that for $n\in\N$ we have \begin{align*}
					\dist(f^{n}(x),f^{n}(y))&< C(\delta)(\lambda+\delta)^{n}\dist(x,y)\text{, for }y\in \StableLocMan{x}\\
					\dist(f^{-n}(x),f^{-n}(y))&< C(\delta)(\mu-\delta)^{-n}\dist(x,y)\text{, for }y\in \UnstableLocMan{x}.
			\end{align*}
			\item\label{theorem:inv-manifolds-hyp-set:tangents:beta} there exist $\beta>0$ and a family of neighbourhoods $O_{x}$ containing the ball around $x\in\Lambda$ of radius $\beta$ such that \begin{align*}
				\StableLocMan{x}&=\{y\in U\colon f^{n}(y)\in O_{f^{n}(x)},~n=0,1,... \}\\
				\UnstableLocMan{x}&=\{y\in U\colon f^{-n}(y)\in O_{f^{-n}(x)},~n=0,1,... \}.
			\end{align*}
		\end{enumerate}
	\end{theorem}
	\begin{remark}
		Note that from the cone criterion \eqref{proposition:cone-criterion-hyp-set:eq:cone}, the $df$-invariance of the decomposition $E^{-}\oplus E^{+}$ and \ref{theorem:inv-manifolds-hyp-set:tangents} follow that $\StableLocMan{x}$ is a $s$-curve and $\UnstableLocMan{x}$ is a $u$-curve, with respect to the families of cones $\{C^{s}_{\rho,z}\}_{z\in M}$ and $\{C^{u}_{\rho,z}\}_{z\in M}$. Moreover, note that there is no unique way of choosing local invariant manifolds.
	\end{remark}

	\subsubsection*{General invariant manifolds and global manifolds}
	One can define invariant manifolds in general, not only for points from a hyperbolic set. Let $g\colon M \to M$ be a diffeomorphism on the surface $M$. Let us recall that the~local \emph{stable manifold} $\StableLocMan{p}$ and the local \emph{unstable manifold} $\UnstableLocMan{p}$ at a point $p$ are defined as follows \begin{align*}
		\StableLocMan{\varepsilon,p} &= \{z \in M\colon \lim_{n\to\infty}d(g^{n}(z),g^{n}(p))=0,~\dist(g^{n}(z),g^{n}(p))\leq \varepsilon\}\text{ and }\\
		\UnstableLocMan{\varepsilon,p} &= \{z \in M\colon \lim_{n\to\infty}d(g^{-n}(z),g^{-n}(p))=0,~\dist(g^{-n}(z),g^{-n}(p))\leq \varepsilon\},
	\end{align*} for some $\varepsilon>0$ small enough. Note that this definition agrees with the one posed for $p$ contained in a hyperbolic set, see condition \ref{theorem:inv-manifolds-hyp-set:tangents:beta} of Theorem \ref{theorem:inv-manifolds-hyp-set}.
	
	The \emph{global invariant manifolds} at $p$ are defined by iterating the local manifolds, that is, \[
	\UnstableMan{p}=\bigcup_{i\geq 0}g^{n}(\UnstableLocMan{p})\text{, and }\StableMan{p}=\bigcup_{i\geq 0}g^{-n}(\StableLocMan{p}).
	\]
	In the context of this article, most of the invariant manifolds we consider will be smooth one-dimensional curves. These manifolds may arise as a consequence of the stable manifold theorem \ref{theorem:inv-manifolds-hyp-set:tangents}, or as subcurves of global manifolds. Nevertheless, some invariant manifolds may not be smooth, if the system is not smooth, or some invariant manifolds may not have tangent vectors contained in the cone' families. Such a phenomenon is crucial in the piecewise hyperbolic or non-uniformly hyperbolic systems and creates rich dynamical behavior seen in systems of this kind.
	
	Note that the definition of invariant manifolds depends on the length of the manifold. We will not signify the length in the notation. That is, if we state that for $\UnstableLocMan{z}$ or $\StableLocMan{z}$ some property holds, we mean that such an invariant local manifold exists for some $\varepsilon$.
	\subsubsection*{Hyperbolic fixed points}
	A fixed point $x\in M$ such that the sequence of differentials $\{df_{f^{n}(x)}=df_{x}\}_{n\in\Z}$ is hyperbolic will be called a \emph{hyperbolic fixed point}. Naturally, the sequence is constant, so its hyperbolicity is equivalent to the fact that the absolute value of no eigenvalue of $df_{x}$ is equal to $1$. Note that in this case, the assumptions of Theorem \ref{theorem:inv-manifolds-hyp-set} are satisfied, that is, a hyperbolic fixed point is a hyperbolic set, and so there are local manifolds $\StableLocMan{x}$ and $\UnstableLocMan{x}$. The manifolds are in fact invariant, that is, \[
	f(\StableLocMan{x})\subset \StableLocMan{x}\text{ and } f^{-1}(\UnstableLocMan{x})\subset \UnstableLocMan{x}.
	\]For the reasoning in this work, we will need a perturbation result that allows the map $f$ to be only an endomorphism, that is, $f$ is allowed not to be injective. Recall from \cite{palistakens:hyp} the following.
	\begin{proposition}[Appendix 4, Thm. 2, \cite{palistakens:hyp}]\label{proposition:inv-manifolds-endo}
		Let $f\colon \R^{2}\to\R^{2}$ be a $C^{1}$ map with a hyperbolic fixed point $p\in\R^{2}$. Then for $\varepsilon>0$ sufficiently small, the local stable manifold is a $C^{1}$ submanifold. Also there is a neighbourhood $\cc{U}$ of $f$ in the $C^{1}$ topology and a continuous map $p\colon \cc{U}\to \R^{2}$ such that for $f'\in\cc{U}$, $p(f')$ is a hyperbolic fixed point of $f'$ and such that the local stable manifold $\StableLocMan{\varepsilon,p(f')}$, as a $C^{1}$ manifold, depends continuously on $f'$, with respect to the $C^{1}$ topology.
	\end{proposition}
	Let us call a $C^{1}$ open disc $U\subset \R^{2}$ any subset diffeomorphic with an open unit ball. One immediately obtains the following.
	\begin{proposition}\label{proposition:inv-manifolds-endo2}
		Let $U\subset \R^{2}$ be a $C^{1}$ open disc. Let $f\colon U\to\R^{2}$ be a $C^{1}$ map with a hyperbolic fixed point $p\in U$. Then for $\varepsilon>0$ sufficiently small, the local stable manifold is a $C^{1}$ submanifold. In addition, there is a neighborhood $\cc{U}$ of $f$ in the $C^{1}$ topology and a continuous map $p\colon \cc{U}\to U$ such that for $f'\in\cc{U}$, $p(f')$ is a hyperbolic fixed point of $f'$ and such that the local stable manifold $\StableLocMan{\varepsilon,p(f')}$, as a $C^{1}$ manifold, depends continuously on $f'$, with respect to the $C^{1}$ topology.
	\end{proposition}
	\section{Perturbations of unimodal maps: admissible families}\label{section:admissible-families}
	\subsection{Definition}
    \begin{definition}
        We will say that a point $q\in \R$ is expanding for some map $f\colon \R\to \R$ if $f$ is $C^1$ in a neighbourhood of $q$ and $|df_q|>1$.
    \end{definition}
	\begin{definition}\label{def:unimodal-map}
		Let $I$ be the compact interval $[q_{-},q_{+}]$ that contains $0$ in the interior and $f\colon \R\to \R$ be a continuous map. If $f$ has a unique maximum at $c=0$, $q_{-}$ is an expanding fixed point, $f(\fr I)\subset \{q_{-}\}$, and $I$ is $f$-invariant, then $f$ will be called a \emph{unimodal map}, and $c$ will be called a \emph{critical} point.
	\end{definition}
	Let $R=I\times I$, $\Rc= \{0\}\times\R$, and $\Rp$ and $\Rn$ be, respectively, right and left connected components of $\plane\setminus\Rc$. The primary objects of this article are $2$-parameter families of homeomorphisms $f_{a,b}\colon \plane\to \plane$, which are diffeomorphisms on $\plane\setminus\Rc$, constructed in the following steps, obtained as perturbations of unimodal families. A similar definition can be found in \cite{StrangeAttractorswithOneDirectionofInstability}.
	\begin{description}
		\item[Step A]\label{stepA} Let $f_{a}\colon \R\to \R$, for $a\in A:=[a_{1},a_{2}]$, be a family of unimodal maps such that \begin{enumerate}[label=(A\arabic*)]
			\item\label{stepA1} $(a,x)\mapsto f_{a}(x)$ is $C^{1}$ on $A\times(\R\setminus\{c\})$,
			\item\label{stepA3} there exists $a_{*}\in(a_{1},a_{2})$ such that $f_{a_{*}}(c)=q_{+}$, and $f_{a}(c)>q_{+}$ for $a\in (a_{*},a_{2})$
			\item\label{stepA2} there exist $\delta_{0}>0$, $\tilde c_{0}>0$, $\tilde c_{1}>0$ and a function $\tilde\theta(\cdot)$ with $\lim_{w\to0}\tilde\theta(w)=0$, such that for every $0<\delta<\delta_{0}$ we have $|df^{n}_{a}(x)|\geq {\tilde c}_{0}\tilde\theta(\delta)e^{\tilde{c}_{1}n}$, whenever $x,...,f^{n-1}(x)\notin B_{\delta}(c)$
		\end{enumerate}
		
		In addition, we also consider the following strengthening of \ref{stepA2}:
		\begin{enumerate}[label=(A3H)]
			\item\label{stepA2-hyp} there exist $\tilde{c}_{1}>0$ and $\alpha>0$ such that $|df^{n}_{a}(x)|\geq e^{\tilde{c}_{1}n}$, for every $n\in\N$, whenever $f_{a}^{i}(x)\neq c$ for $i\in\N$ and $|a-a_{*}|< \alpha$,
		\end{enumerate}and the high expansion condition, necessary for our proof of mixing:	\begin{enumerate}[label=(A3E)]
			\item\label{stepA2-expansion} $e^{\tilde c_{1}}>\sqrt{2}$
		\end{enumerate}
		\item[Step B]\label{stepB}			
		Let $f_{a,b}\colon \plane\to\plane$, for $b\in B:=[0,b_{1}]$, where $0<b_1<1$, be a family of homeomorphisms on $\plane$ such that
		\begin{enumerate}[label=(B\arabic*)]
			\item\label{stepB1} $(a,b,z)\mapsto g_{a,b}(z):=(f_{a,b}-f_{a,0})(z)$ is $C^{1}$ on $ A\times B\times \plane$,
			\item\label{stepB11} $f_{a,b}$ is an extension of $(f_{a}(x),0)\colon \R\to\R\times\{0\}$, that is we have $f_{a,0}(x,y)=(f_{a}(x),0)$,
			\item\label{stepB2} for every compact $K\subset \plane$ there exists a constant $C>0$ independent of parameters such that $||g_{a,b}|_{K}||_{C^{1}}<Cb$.
		\end{enumerate}
	\end{description}
	\begin{remark}\label{remark:stepsA-B}
		Note that the existence of a positive Lyapunov exponent is immediately implied by \ref{stepA2-hyp}. Moreover, by the mean value theorem, condition \ref{stepB2} is always fulfilled whenever we assume additionally that $(x,y,a,b)\mapsto g_{a,b}(x,y)$ is $C^{2}$ on $ A\times B\times \plane$.
	\end{remark}
	
	\begin{definition}
		A family $\{f_{a,b}\}_{a\in A,b\in B}$ as in \ref{stepA1}-\ref{stepA3} and \ref{stepB1}-\ref{stepB2} will be called \emph{an admissible family of perturbations of a unimodal family} $\{f_{a}\}_{a\in A}$, or an admissible family, for short. An admissible family satisfying \ref{stepA2-hyp} will be called \emph{piecewise uniformly hyperbolic}. If we additionally assume \ref{stepA2-expansion}, we will say that the family has a high expansion.
	\end{definition}
	
	\begin{convention}
		From now on, we will reserve $\{f_{a,b}\}_{a\in A,b\in B}$ for an admissible family. By $\{f_a\}_{a\in A}$ we will denote the one dimensional family yielding the admissible family. Moreover, we will often drop the dependence on $a,b$ and write $f$ for $f_{a,b}$. Thus, whenever we write $f_a$, we mean an element of the one dimensional family, and by $f$ we mean an element of the admissible family, and by $f_{a,0}$ we mean the endomorphism $(f_{a},0)$.
		
		The constants $\tilde{c}_{0},\tilde{c}_{1}$ and function $\tilde{\theta}$ will be fixed from now on.
		
		Moreover, whenever we will say that for $a \in A'$ and $b\in B'$, for some $A'\subset A$, $B'\subset B$, some conditions hold, we will mean that they hold for parameters $(a,b)\in A'\times B'$. 
	\end{convention}
	\begin{remark}
		The conditions contained in Steps A. and B. are satisfied by many natural families of planar maps, such as family of border-collision normal forms, Lozi-like families, H\'enon family, or more general diffeomorphic $C^{2}$ perturbations of Misiurewicz maps satisfying \ref{stepA1},\ref{stepA2}, with negative Schwarzian derivative, non-flat critical point, and hyperbolic repelling periodic points. The case of border-collision normal forms is discussed later in the article. The Wang-Young families are admissible families as well, see \cite{StrangeAttractorswithOneDirectionofInstability} for definitions.
	\end{remark}
	Let $\Ryc=f(\Rc)$. Define $C_{\delta}:=\{z\in\plane\colon d(z,\Rc)<\delta\}$. Let $\pi_1(z)=(z)_1=x$, for $z=(x,y)$, be the projection onto the first variable. We have $\pi_1\colon \R^2\to \R^2$, and by basic identification, also $\pi_1\colon T_z \R^2\to \R$, where $T_z \R^2$ is the tangent space to $\R^2$ at $z\in \R^2$.
    
	We have the following basic observation.
	\begin{remark}\label{remark:phic}
		As the curve $\Ryc$ is a $u$-curve for $b>0$ sufficiently small, define $\phi_{c}\colon \R\to \R$ to be the function which parameterizes $\Ryc$ and satisfies $\frac{d}{dt}\phi_c> 0$. 
	\end{remark}
	
	\begin{convention}\label{definition:directions-on-sets}
		We define a partial order on all subsets of $\plane$, by saying that $E$ is \emph{to the left of} $D$, equivalently, $D$ is \emph{to the right of} $E$, if $x \leq y$ for $(x,\phi_{c}(x))\in E\cap \Ryc$ and $(y,\phi_{c}(y))\in D\cap \Ryc$.  In a more compact notation, we have $\pi_{1}(E\cap \Ryc)\leq \pi_{1}(D\cap \Ryc)$, where $\pi_{1}\colon \plane\to\R$ is the projection on the first variable.
	\end{convention}
	\begin{figure}[!ht]
		\centering
		\resizebox{0.7\textwidth}{!}{%
			\begin{circuitikz}
				\tikzstyle{every node}=[font=\tiny]
				\draw [->, >=Stealth] (5,12) -- (12,12);
				\draw [->, >=Stealth] (8.5,8.5) -- (8.5,15.25);
				\draw [short] (5.5,9) .. controls (6,11.5) and (6.25,12.5) .. (8.5,14);
				\draw [short] (8.5,14) .. controls (10.5,12.5) and (10.75,11.75) .. (11.25,9.25);
				\draw [dashed] (5.75,15) -- (5.75,9.25);
				\draw [dashed] (10.75,15.25) -- (10.75,9.25);
				\node at (5.75,12) [circ] {};
				\node at (10.75,12) [circ] {};
				\node [font=\tiny] at (5.3,11.75) {$q_-=z_-$};
				\node [font=\tiny] at (11,11.5) {$q_+$};
				\draw [short] (5.25,9.75) -- (12.5,14.75);
				\node [font=\tiny] at (12.25,14.75) {$x=y$};
				\node [font=\tiny] at (6.75,13) {$y=f_a(x)$};
				\draw [dashed] (9.75,12.75) -- (9.75,12);
				\node at (9.75,12) [circ] {};
				\node [font=\tiny] at (9.75,11.75) {$z_+$};
			\end{circuitikz}
		}%
		
		\label{fig:stepA}
	\end{figure}
	
	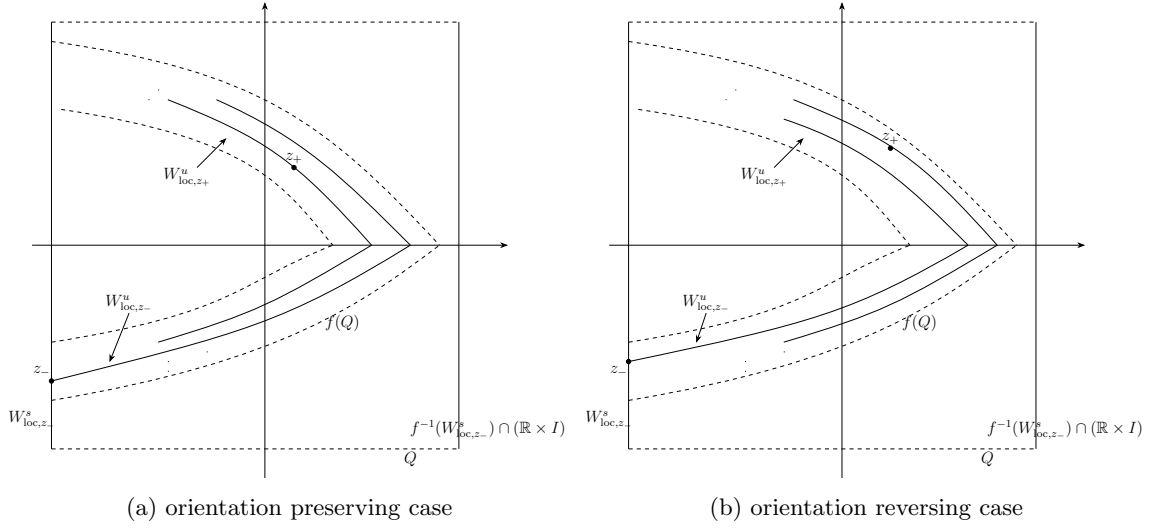
\begin{figure}
		\begin{subfigure}{.5\textwidth}
			\centering
			\resizebox{\textwidth}{!}{%
				\begin{circuitikz}
					\tikzstyle{every node}=[font=\large]
					\draw [->, >=Stealth] (1.5,11.75) -- (13.75,11.75);
					\draw [->, >=Stealth] (7.5,5.75) -- (7.5,18);
					\draw [short] (2,17.5) -- (2,6.5);
					\draw [short] (12.5,17.5) -- (12.5,6.5);
					\draw [dashed] (2,17.5) -- (12.5,17.5);
					\draw [dashed] (2,6.5) -- (12.5,6.5);
					\node at (2,8.25) [circ] {};
					\node at (8.25,13.75) [circ] {};
					\draw [short] (2,8.25) .. controls (7,9.5) and (8,9.75) .. (11.25,11.75);
					\draw [short] (11.25,11.75) .. controls (9.5,13.5) and (8.5,14.5) .. (6.25,15.5);
					\draw [short] (5,15.5) .. controls (7.5,14.5) and (8.25,14) .. (10.25,11.75);
					\draw [short] (10.25,11.75) .. controls (7.75,10.25) and (7.25,10) .. (4.75,9.25);
					\draw [dashed] (2,17) .. controls (7.25,16.25) and (9,15) .. (12,11.75);
					\draw [dashed] (4.5,15.5) -- (4.5,15.5);
					\draw [dashed] (4.75,15.75) -- (4.75,15.75);
					\draw [dashed] (2,7.75) .. controls (8,8.75) and (10,10.25) .. (12,11.75);
					\draw [dashed] (6,9) -- (6,9);
					\draw [dashed] (5,8.75) -- (5,8.75);
					\draw [dashed] (5,8.5) -- (5,8.5);
					\draw [dashed] (2.25,15.25) .. controls (7.25,14.5) and (8,13.25) .. (9.25,11.75);
					\draw [dashed] (2,9.25) .. controls (6.75,10) and (6.75,10.75) .. (9.25,11.75);
					\node at (1.75,8.5) {$z_-$};
					\node  at (8.25,14) {$z_+$};
					\draw [->, >=Stealth] (5.75,13.75) -- (6.5,14.5);
					\draw [->, >=Stealth] (4,10) -- (3.5,8.75);
					\node  at (4,10.25) {$\UnstableLocMan{z_-}$};
					\node  at (5.5,13.5) {$\UnstableLocMan{z_+}$};
					\node  at (9.5,9.75) {$f(Q)$};
					\node  at (11.25,6.25) {$Q$};
					\node  at (1.5,7.25) {$\StableLocMan{z_-}$};
					\node  at (13.25,7) {$f^{-1}(\StableLocMan{z_-})\cap (\R\times I)$};
				\end{circuitikz}
			}%
			\caption{orientation preserving case}
			\label{fig:manifolds-positions-OP}
		\end{subfigure}
		\begin{subfigure}{.5\textwidth}
			\centering
			\resizebox{\textwidth}{!}{%
				\begin{circuitikz}
					\tikzstyle{every node}=[font=\large]
					\draw [->, >=Stealth] (1.5,11.75) -- (13.75,11.75);
					\draw [->, >=Stealth] (7.5,5.75) -- (7.5,18);
					\draw [short] (2,17.5) -- (2,6.5);
					\draw [short] (12.5,17.5) -- (12.5,6.5);
					\draw [dashed] (2,17.5) -- (12.5,17.5);
					\draw [dashed] (2,6.5) -- (12.5,6.5);
					\node at (2,8.75) [circ] {};
					\node at (8.75,14.25) [circ] {};
					\draw [short] (2,8.75) .. controls (6.75,9.75) and (7.75,10) .. (10.75,11.75);
					\draw [short] (10.75,11.75) .. controls (9.25,13.25) and (8.25,14.25) .. (6,15);
					\draw [short] (6.25,15.5) .. controls (8.75,14.5) and (9.5,14) .. (11.5,11.75);
					\draw [short] (11.5,11.75) .. controls (9,10.25) and (8.5,10) .. (6,9.25);
					\draw [dashed] (2,17) .. controls (7.25,16.25) and (9,15) .. (12,11.75);
					\draw [dashed] (4.5,15.5) -- (4.5,15.5);
					\draw [dashed] (4.75,15.75) -- (4.75,15.75);
					\draw [dashed] (2,7.75) .. controls (8,8.75) and (10,10.25) .. (12,11.75);
					\draw [dashed] (6,9) -- (6,9);
					\draw [dashed] (5,8.75) -- (5,8.75);
					\draw [dashed] (5,8.5) -- (5,8.5);
					\draw [dashed] (2.25,15.25) .. controls (7.25,14.5) and (8,13.25) .. (9.25,11.75);
					\draw [dashed] (2,9.25) .. controls (6.75,10) and (6.75,10.75) .. (9.25,11.75);
					\node  at (1.75,8.5) {$z_-$};
					\node  at (8.75,14.5) {$z_+$};
					\draw [->, >=Stealth] (5.75,13.75) -- (6.5,14.5);
					\draw [->, >=Stealth] (4,10) -- (3.75,9.25);
					\node  at (4,10.25) {$\UnstableLocMan{z_-}$};
					\node  at (5.5,13.5) {$\UnstableLocMan{z_+}$};
					\node  at (9.5,9.75) {$f(Q)$};
					\node  at (11.25,6.25) {$Q$};
					\node  at (1.5,7.25) {$\StableLocMan{z_-}$};
					\node  at (13.25,7) {$f^{-1}(\StableLocMan{z_-})\cap (\R\times I)$};
				\end{circuitikz}
			}%
			\caption{orientation reversing case}
			\label{fig:manifolds-positions-OR}
		\end{subfigure}
		\caption{Positions of invariant manifolds and the regions $Q$ and its image $f(Q)$}
	\end{figure}
	\subsection{Hyperbolicity}
	For a fixed $\delta>0$ and compact $K\subset \plane$, let us define \[S(\delta):=\sup\{|(df_{a,0,z})_{1}|\colon a\in A ,~z\in K\setminus C_{\delta}\}.\]
	\begin{lemma}\label{lemma:cone-inv}
		Let $K\subset \plane$ be compact. Let $0<c_{0}<\tilde{c}_{0}$, and $0<c_{1}<\tilde{c}_{1}$. Then there exist $\delta_{0}>0$, $\theta\colon [0,\infty)\to[0,\infty)$, with $\lim_{q\to0}\theta(q)=0$, and $c_{2}>0$ such that for any $0<\delta<\delta_{0}$, $n\in\N$, there is $b_{2}>0$ such that for $b\in(0,b_{2})$, $c_{2}b\leq\rho\leq 1$, and $z\in K$ with $f^{i}(z)\in K\setminus C_{\delta}$, for $i=0,...,n-1$, we have \begin{equation}
			\label{lemma:cone-inv:C1}  df^{i}_{z}(C^{u}_{z,\rho})\subset \inter C^{u}_{f^{i}(z),\rho} \text{ and } ||df^{i}_{z}(v)||\geq c_{0}\theta e^{c_{1}i}||v||  \text{ for every } v\in C^{u}_{z,\rho}.
		\end{equation}Moreover, for any $\mu>0$ and $\delta>0$, there is $b_{3}>0$, such that for any $0<b<b_{3}$ and $z\in f(K\setminus C_{\delta})$ we have \begin{align}\label{lemma:cone-inv:C2}
			df^{-1}_{z}(C^{s}_{z,\rho})\subset \inter C^{s}_ {f^{-1}(z),\rho}&\text{ for}~0<\rho<1, \text{ and }  \\\notag||df^{-1}_{z}(w)||\geq \mu||w||\text{ for every }w\in C^{s}_{z,\rho}&\text{ for }0<\rho\leq(\mu^{-1}-Cb)/S,
		\end{align}where $S:=S(\delta)$ and $C>0$ is the constant from \ref{stepB2}.
	\end{lemma}
	\begin{proof}
		Let $0<\rho<1$ and $\delta_0>0$ be given by the definition of admissible families. Let $0<\delta<\delta_0$. Fix $b_{11}>0$ small enough so that $\tilde\theta (1+\rho^{2})^{-1/2}\tilde{c}_{1} e^{\tilde{c}_{2}}-b_{11}>0$ for $\tilde\theta:=\tilde\theta(\delta)$. Choose $N\in\N$ large enough for $(1+\rho^{2})^{-1/2}\tilde\theta \tilde{c}_{1} e^{\tilde{c}_{2}N}-b_{11}>e^{c_{2}N}$. Now, let us find $0<b_{12}<b_{11}$ so that for every $b\in(0,b_{12})$ and $0\leq i\leq N$ we have \begin{equation}\label{eq:lemma:cone-inv:1}
			||df^{i}-df^{i}_{a,0}||< b_{11}.
		\end{equation}
		Let $z=(x,y)\in K$ and consider $C^{u}_{z,\rho}$. The idea of the proof is to use perturbation to show properties \eqref{lemma:cone-inv:C1} separately for $df^{r}$, $1\leq r< N$, and $df^{N}$, and reason \eqref{lemma:cone-inv:C1} for any $n=r+pN$, $p\in\N$. First, note that \eqref{eq:lemma:cone-inv:1} implies expansion under $1\leq r\leq N$. Indeed, for any $v\in C^{u}_{\rho}$, $||v||=1$, we have \begin{align}\label{eq:lemma:cone-inv:2}
			||df^{r}v||\geq ||df^{r}_{a,0}v||-b_{11}=|df^{r}_{a}(x)||v_{1}|-b_{11}\geq
			|df_{a}^{r}(x)|(1+\rho^{2})^{-1/2}-b_{11}\geq\\\notag(1+\rho^{2})^{-1/2}\tilde{\theta}\tilde{c}_{0}e^{\tilde{c}_{1}r}-b_{11}\geq c_{0}e^{c_{1}r}\left((1+\rho^{2})^{-1/2}\tilde{\theta}-b_{11}\tilde{c}_{0}^{-1}e^{-\tilde{c}_{1}r}\right).
		\end{align} Therefore, \eqref{lemma:cone-inv:C1} holds for $1\leq n<N$ and  $\theta=(1+\rho^{2})^{-1/2}\tilde{\theta}-b_{11}\tilde{c}_{0}^{-1}e^{-\tilde{c}_{1}r}$. In the case $r=N$, by the choice of $N$, we have $||df^{N}v||\geq e^{c_{1}N}$.
		
		We now deal with the invariance of unstable cones under $f$. Let $v\in C^{u}_{\rho}$ with $\rho\geq c_{2}b$, for some $c_{2}>0$ to be fixed later. First, note that by definition $||df-df_{a,0}||<Cb$, for some $C>0$. Hence, we can compute similarly to \eqref{eq:lemma:cone-inv:2} \begin{multline*}
			|(dfv)_{1}|\geq (1+\rho^{2})^{-1/2}\tilde{\theta}\tilde{c}_{0}e^{\tilde{c}_{1}}-Cb\geq (1+\rho^{2})^{-1/2}\left(\tilde{\theta}\tilde{c}_{0}e^{\tilde{c}_{1}}-(1+\rho^{2})^{1/2}Cb\right)\geq\\ (1+\rho^{2})^{-1/2}\left(\tilde{\theta}\tilde{c}_{0}e^{\tilde{c}_{1}}-2Cb\right)=:(1+\rho^{2})^{-1/2}{c'_{2}}^{-1},
		\end{multline*}for $b$ small enough and $c'_{2}>0$ dependent on $\delta,\tilde{c}_{0},\tilde{c}_{1}$, and defined in the last equality. Then, the definition of admissible families implies $|(dfv)_{2}|\leq Cb$, and so\[|(dfv)_{2}|/|(dfv)_{1}|\leq Cb(1+\rho^{2})^{1/2}{c'}_{2}< bc_{2}\leq\rho
		\] for $b$ small enough, $c_{2}=2Cc'_{2}$ and $\rho\geq bc_{2}$, showing the invariance of the unstable cones.
		
		We now show expansion under $pN$, for any $p\in\N$. By induction on $p$, using the invariance under $df^{N}$, we have \[
		||df^{pN}v||\geq ||df^{(p-1)N}(df^{N}v)||||df^{N}v||\geq  e^{c_{1}(p-1)N}e^{c_{1}N}= e^{c_{1}pN}.
		\]It remains to prove expansion under any $n=r+pN$, for some $0\leq r<N$ and $p\in\N$. We have, since $df^{pN}v\in C^{u}$, \[
		||df^{n}v||\geq ||df^{r}(df^{pN}v)||||df^{pN}v||\geq c_{0}\theta e^{c_{1}n},
		\]which concludes the proof of \eqref{lemma:cone-inv:C1}.
		
		Let us show the last claim of the lemma. We first show the inclusion. Let $v\in T_{z}\plane$ and $z\notin C_{\delta}$. Then, as in the beginning of the proof, we have \[
		|(dfv)_{1}|\geq |(df_{a,0}v)_{1}| - |(dgv)_{1}|\geq L|v_{1}|-Cb||v||,
		\] where $L:=\inf\{|(df_{a,0,z})_{1}|\colon a\in A ,~z\in K\setminus C_{\delta}\}.$ Dividing the inequality by $L||v||$, we obtain \begin{equation}\label{eq:lemma:cone-inv:2:eq:2}
			|v_1|/||v||\leq \left(Cb+|(dfv)_{1}|/||v||\right)/L.
		\end{equation}
		By the definition of admissible family, we have $|(dfv)_{2}|=|(dgv)_{2}|\leq Cb||v||$. Hence, $ ||v||\geq |(dfv)_{2}|C^{-1}b^{-1}$. Substituting this into \eqref{eq:lemma:cone-inv:2:eq:2}, we have \[
		|v_1|/||v||\leq \frac{Cb}{L} (1+|(dfv)_{1}|/|(dfv)_{2}|).
		\]
		Let us now put $v=df^{-1}w$, for $w\in C^{s}_{z,\rho}$ and $z\in f(K\setminus C_{\delta})$. Using $|w_1|\leq \rho |w_2|$, we see \[
		\dfrac{|(df^{-1}w)_{1}|}{||df^{-1}w||}\leq bC(1+\rho)/L\leq b/\tilde L,
		\]where $\tilde L:=LC^{-1}(1+\rho)^{-1}$. Now, \[|(df^{-1}w)_{1}|^{2}=||df^{-1}w||^{2}-|(df^{-1}w)_{2}|^{2}\geq |(df^{-1}w)_{1}|^{2}b^{-2}\tilde L^2-|(df^{-1}w)_{2}|^{2},\] and so \begin{equation}\label{lemma:cone-inv:eq:1}
			\dfrac{|(df^{-1}w)_{1}|}{|(df^{-1}w)_{2}|}\leq(b^{-2} \tilde L^2-1)^{-1/2}\leq b\tilde L^{-1}(1-b^{2}\tilde L^{-2})^{-1/2}\rho,
		\end{equation}
		for $b$ small enough and $0<\rho<1$. Note that since $\tilde L$ is bounded away from zero uniformly in parameters, \eqref{lemma:cone-inv:eq:1} is strictly bounded by $\rho$ for all $b$ small enough dependent only on $\rho, \delta$ and $K$. This shows the invariance.
		
		Now, let $\rho>0$ be close to $0$. Let $df^{-1}w=v\in C^{s}_{\rho}$ for some $w\in C^{s}_{\rho}$. Recall $S:=\sup\{|(df_{a,0,z})_{1}|\colon a\in A ,~z\in K\setminus C_{\delta}\}$. We compute, using the invariance of stable cones \[
		||dfv||\leq ||df_{a,0}v||+Cb||v||\leq S |v_1|+Cb||v||\leq S\rho |v_2|+Cb||v||\leq (S\rho+Cb)||v||.\]Thus, for $b$ and $\rho$ small enough we have $||dfv||\leq \mu^{-1} ||v||$, which concludes the proof.
	\end{proof}
	\begin{remark}\label{remark:what-are-stablevectorsetc}
		From now on, whenever we call a vector $v$ almost horizontal, we will mean a vector contained in an unstable $c_{2}b$-cone. By a $u$-curve we mean a curve with tangent vectors being almost horizontal. Similarly, in the case of stable cones, we require that almost vertical vectors and $s$-curves be contained in stable $b$-cones. We see that with such a convention, by Lemma \ref{lemma:cone-inv}, almost horizontal vectors are expanded by $df$, and almost vertical vectors are expanded by $df^{-1}$.
	\end{remark}
	\begin{proposition}\label{proposition:hyperbolic-set-admissible-family}
		For any compact $K$ there exist $\delta_{0}>0$ and $b_{0}>0$ such that for any $0<\delta<\delta_{0}$ and $0<b<b_{0}$ the set $\cc{H}\subset K$ of $z\in K$ with $f^{i}(z)\in K\setminus C_{\delta}$, for $i\in \Z$, is a hyperbolic set of $f$. Moreover, for any such $z$ there exists its stable local manifold that is almost vertical.
	\end{proposition}
	\begin{proof}
		First, note that $\cc{H}$ is compact and invariant. Indeed, since $C_{\delta}$ is open, $K\setminus C_{\delta}$ is compact and so $\cc{H}=\bigcap_{i\in\Z}f^{i}(K\setminus C_{\delta})$ is compact as well. The invariance of $\cc{H}$ is immediate. Let us show hyperbolicity.
		
		We choose $\delta_{0}>0$ and $b_{0}$, so that for $0<\delta<\delta_{0}$ and $0<b<b_{0}$ the cones' invariance and expansion conditions \eqref{lemma:cone-inv:C1}, \eqref{lemma:cone-inv:C2} of Lemma \ref{lemma:cone-inv} are satisfied. Then for any $z\in \cc{H}$, from the definition of $\cc{H}$ follows that \eqref{lemma:cone-inv:C1} and \eqref{lemma:cone-inv:C2} hold. This means that the assumptions of Proposition \ref{proposition:cone-criterion-hyp-set} are satisfied, and so $\cc{H}$ is a hyperbolic set.
		
		It remains to show that any $z\in\cc{H}$ has some almost vertical stable manifold. Note that the invariance of stable cones condition \eqref{lemma:cone-inv:C2} holds with $\rho= b$, if $b$ is small enough. Moreover, according to Lemma \ref{lemma:cone-inv} the invariant subspace $E_{z}^{-}=T_{z}\StableLocMan{z}$ is contained in $\inter C^{s}_{\rho, z}$, and so by the fact that $\StableLocMan{z}$ is $C^{1}$, for $z'\in \StableLocMan{z}$ close enough to $z$ we have $T_{z'}\StableLocMan{z}\subset \inter C^{s}_{\rho, z'}$. This means that there exists a stable local manifold of $z$ that is almost vertical.
	\end{proof}

	\begin{remark}\label{remark:case-b-0}[The case of $b=0$]
		Let us note that in the case of $b=0$, all local stable manifolds of points $z\in \cc{H}$, defined in Proposition \ref{proposition:hyperbolic-set-admissible-family}, are vertical segments intersecting $\Ryc$. Indeed, this is a consequence of the fact that $f_{a,0}$ projects all points onto $\Ryc$.
	\end{remark}
	\subsection{Basic properties}
	We denote by $\pi_{1}(\cdot)=(\cdot)_{1}$ and $\pi_{2}(\cdot)=(\cdot)_{2}$ the projections on the first and second variables. 
	
	\begin{convention}
		We will write $A^{b} \nearrow (a',a_{*})$, for some $a'\in A$, if $A^{b}\subset A^{b'}$ for $b\leq b'$ and $\bigcup_{b}A^{b}=(a',a_{*})$.
	\end{convention}
	\begin{remark}
		The one dimensional maps $f_{a}$ have two fixed points, one equal to $q_{-}$, and the other contained in $[0,q_+]$. By abuse of notation, we will denote the fixed points of $f_{a}$ by, respectively, $z_-$ and $z_+$. Those fixed points are hyperbolic and so persist under perturbations. We will denote by the same notation their hyperbolic continuations for $f_{a,b}$, $b>0$ small. The reader will know from the context whether the fixed points $z_-,z_+$ correspond to the one dimensional maps $f_{a}$ or their perturbations $f_{a,b}$. 
	\end{remark}
	The reader is advised to keep in mind Figure \ref{fig:renormalization}, as it shows the relative positions of invariant manifolds of the fixed points. We will assume that $f$ maps the left half plane to the component of $f(\plane)\setminus\Ryc$ below $\Ryc$, and the right half plane to the component above. Let $Q$ be the connected bounded component of $E:=(\R\times I)\setminus(\StableLocMan{z_{-}}\cup f^{-1}(\StableLocMan{z_{-}})\cap \Rp)$.  Note that $Q$ is well-defined. Indeed, as $(a,b)\to (a,0)$, $\StableLocMan{z_{-}}\cup f^{-1}(\StableLocMan{z_{-}})\cap \Rp \to \{q_{-}\}\times \R\cup \{q_{+}\}\times \R$, in the Hausdorff metric. Therefore, $Q\to[q_{-},q_{+}]\times I$, which is bounded. In the following lemma, we gather some basic facts about admissible families. 
	\begin{lemma}\label{lemma:basic-properties}	There exists $b_{0}\in B$ such that for $b\in B_{0}:=(0,b_{0})$ there are intervals $A^{b} \nearrow (a',a_{*})$, for some $a'\in A$, such that the following hold for $(a,b)\in A^{b}\times\{b\}$
		\begin{enumerate}
			\item\label{lemma:basic-properties:prop1} for $b$ small enough $Q$ is a forward invariant rectangle
			\item\label{lemma:basic-properties:prop2} $f$ maps the left and right half planes diffeomorphically to the connected components of $\plane\setminus\Ryc$, and $\gamma:=f(Q)\cap \Ryc$ is a curve in $\Ryc$ with $\length(\gamma)\to 0$ as $b\to0$
			\item\label{lemma:basic-properties:prop3} $f$ has two unique hyperbolic fixed points, $z_{-}$ in the left half plane,  and the other, $z_{+}$, in the right half plane
			\item\label{lemma:basic-properties:prop4} $\StableLocMan{z_{+}}$ is an $s$-curve on $Q$, and $ f(Q)\setminus \StableLocMan{z_{+}} $ has three connected components, one in the upper half plane, one in the lower half plane, and one, containing $f(Q)\cap \Ryc$ in the right half plane.
		\end{enumerate}
	\end{lemma}
	\begin{proof}
		The properties \ref{lemma:basic-properties:prop1}-\ref{lemma:basic-properties:prop3} follow directly from the corresponding properties of one dimensional family $f_{a}$ by perturbation. The property \ref{lemma:basic-properties:prop4} follows by a perturbation from Proposition \ref{proposition:inv-manifolds-endo2} and the fact that $\StableLocMan{z_{+}}$ is vertical for $f_{a,0}$, see Remark \ref{remark:case-b-0}. The parameter $b_{0}$ is chosen so that the perturbative argument is valid. That is, for all $a\in A$ and $b\in B_{0}$ we have $f(Q)\subset Q$, since this is feasible for a one dimensional family, it is also feasible for its perturbation.
	\end{proof}
	Let us define $\epsilon=1$, if $f$ preserves the orientation and  $\epsilon=-1$ otherwise. Depending on the sign of $\epsilon$, the relative positions of invariant manifolds of $z_{-}$ and $z_{+}$ change. To cover both cases of $\epsilon\in\{-1,1\}$, we will sometimes write $z_{-\epsilon}$ or $z_{\epsilon}$, to indicate one of the fixed points. The following lemma summarizes the positions of invariant manifolds of the fixed points, and the action of $f$ restricted to those manifolds.
	\begin{lemma}\label{lemma:position-inv-manifolds-fixed-points}
		The following holds for $(a,b)\in A^{b}\times\{b\}$ and $b$ small enough, for $\epsilon\in \{-1,1\}$, \begin{enumerate}
			\item\label{lemma:position-inv-manifolds-fixed-points:1} $f|_{\StableLocMan{z_{\epsilon}}}$ and $f|_{\UnstableLocMan{z_{+}}}$ reverse orientation, and $f|_{\UnstableLocMan{z_{-}}}$, $f|_{\StableLocMan{z_{-\epsilon}}}$ preserve orientation,
			\item\label{lemma:position-inv-manifolds-fixed-points:2} 
            $\UnstableLocMan{z_{\epsilon}}\cap\Ryc $ is to the left of $\UnstableLocMan{z_{-\epsilon}}\cap \Ryc$
			\item\label{lemma:position-inv-manifolds-fixed-points:3} there are $\UnstableLocMan{z_{+}}$ and $\UnstableLocMan{z_{-}}$ intersecting each of $\Rc$, and $\StableLocMan{z_{+}}\cap (\R\times I)$ at exactly two points,
		\end{enumerate}
	\end{lemma}
	\begin{proof}
		Claim \ref{lemma:position-inv-manifolds-fixed-points:1} and is a direct consequence of the definition of admissible families, see \cite{kucharski:str-att,strange-attractor-mis} for details in the case of piecewise affine maps. Regarding Claim \ref{lemma:position-inv-manifolds-fixed-points:2}, let us consider a disc $Q\cap \Rn$ and orientation on $F:=\fr(Q\cap \Rn)$ that agrees with the orientation on $\Rc$ induced from $\R$ using the parametrization $\Rc=\{(\psi(x),x)\colon x\in\R\}$ with some $\psi\in C^1(\R,\R)$ and $\frac{\diff}{\diff x}\psi>0$, that is parametrization that goes from down to up. Then $f$ preserves orientation of $F$ if $\epsilon=1$, and reverses otherwise. This means, that every two points $z_1,z_2\in F\cap \Rc$ with $\pi_2(z_1)>\pi_2(z_2)$, have $\epsilon\pi_2(f(z_1))<\epsilon\pi_2(f(z_2))$. Applying this to $z_1=\UnstableLocMan{z_
        +}$ and $z_2=\UnstableLocMan{z_-}$, we obtain Claim \ref{lemma:position-inv-manifolds-fixed-points:2}.

        Let us explain claim \ref{lemma:position-inv-manifolds-fixed-points:3}. Taking into account Lemma \ref{lemma:basic-properties}, it is enough to show that both $\UnstableLocMan{z_{+}}$ and $\UnstableLocMan{z_{-}}$ intersect $\Rc$. We show only the case of $z_{-}$, as the case of $z_{+}$ is analogous. Assume $\UnstableMan{z_{-}}$ does not intersect $\Rc$. Hence, there is $\delta>0$ with $\UnstableMan{z_{-}}\cap C_{\delta}=\emptyset$. Then by Lemma \ref{lemma:cone-inv}, the tangent vectors $v$ to $\UnstableMan{z_{-}}$ are contained in the unstable cones, and so have angles with $(0,1)$ uniformly bounded away from zero. This implies that the projections $\pi_{1}(v)$ on the first coordinate of $v$ are uniformly bounded away from zero and, consequently, $\UnstableMan{z_{-}}$ necessarily have a finite length. On the other hand, by Lemma \ref{lemma:cone-inv}, $\UnstableLocMan{z_{-}}$ is uniformly expanded under $f$, yielding a contradiction.
	\end{proof}
	Let us remark that $\UnstableLocMan{z_{+}}$ and $\UnstableLocMan{z_{-}}$ are folded under the action of $f$ into a $U$-shape or a $V$-shape if they cross $\Rc$, and so, if they have appropriate length, intersect any $s$-curve on $C_{\delta}\cap\Rp$ at exactly two points.
	
	Let $a_{*,b}\in A$ be defined so that at the parameter $p=(a_{*,b},b)$ the map $f_{p}$ has a heteroclinic tangency of $f^{-1}(\StableLocMan{z_{-}})\cap \Rp$ and $\UnstableLocMan{z_{+}}$ (if $f_{p}$ is orientation reversing) or a homoclinic tangency of $f^{-1}(\StableLocMan{z_{-}})\cap \Rp$ and $\UnstableLocMan{z_{-}}$ (if $f_{p}$ is orientation preserving).
	
	\begin{remark}\label{remark:endpoint-of-Ab}
		We may extend the sets $A^{b}$ so that the right endpoint of $A^{b}$ is equal to $a_{*,b}$. From now on, we will assume that this is the case.
	\end{remark}

	\section{Piecewise uniformly hyperbolic families}\label{section:piecewise-uni-hyp-families}
    In this section, we assume that the admissible family is piecewise hyperbolic and has a high expansion, that is, conditions \ref{stepA2-hyp} and \ref{stepA2-expansion} hold.
	\subsection{Chaotic properties}\label{subsection:chaotic-prop}
	In this section, we show that the stable manifold of $z_{+}$ is dense in $Q$ for some parameters near $(a_{*},0)$. The proof essentially does not differ from that presented in \cite{Lozi-likemaps,strange-attractor-mis}, although we present more details. Let $\rho$ be the coefficient of the unstable cones. In this section, we assume $(a,b)\in A^{b}\times\{b\}$, so that Lemmas \ref{lemma:basic-properties} and \ref{lemma:position-inv-manifolds-fixed-points} hold. We begin the proof by the following lemma. 
	\begin{lemma}\label{lemma:smooth-arcs-bounded}
		Any smooth $u$-curve $\gamma\subset Q$ satisfies $\length(\gamma)\leq \sqrt{1+\rho^{2}} \diam Q$.
	\end{lemma}
	\begin{proof}
		Since $\gamma$ is a $u$-curve, it can be parameterized by a segment $J:=\pi_{1}(\gamma)$, that is, one can find $\phi\colon J\to\R$ such that $\gamma(t)=(t,\phi(t))$, for $t\in J$. It follows $\gamma'(t)=(1,\phi'(t))$ for $t\in J$. Recall $\gamma'\subset  C^{u}_{\rho}$. Let us compute \begin{equation*}
			\length(\gamma)=\int_{t\in J}||\gamma'(t)||dt\leq \int_{t\in J} \sqrt{1+\rho^{2}}dt\leq |J| \sqrt{1+\rho^{2}}\leq \diam Q \sqrt{1+\rho^{2}}.\qedhere
		\end{equation*}
	\end{proof}
	\begin{lemma}\label{proposition:st-mani-X-intersects-arcs}
		The stable manifold of $z_{+}$ intersects any $u$-curve $\gamma\subset Q$.
	\end{lemma}
	\begin{proof}Let $\gamma\subset Q$ be a smooth $u$-curve. By condition \ref{lemma:basic-properties:prop4} of Lemma \ref{lemma:basic-properties} it suffices to prove that $f^{n}(\gamma)$ intersects both $\Rc$ and $\Ryc$ simultaneously for $n\in\N$ large enough, since then it connects the components of $f(Q)\setminus \StableLocMan{z_{+}}$ and therefore must intersect $\StableLocMan{z_{+}}$. Assume $f^{n}(\gamma)$ never intersects $\Rc$ and $\Ryc$ at the same time. The iteration $f^{2n}(\gamma)$ will consist of at most $2^{n}$ smooth curves, each with tangent vectors contained in $C^{u}$. Hence, by Lemma \ref{lemma:smooth-arcs-bounded}, $\length(f^{2n}(\gamma))\leq 2^{n}\sqrt{1+\rho^{2}}\diam Q$. On the other hand, since $\gamma'$ is contained in $C^{u}$, we get $\length(f^{2n}(\gamma))\geq e^{2nc_{1}}\length(\gamma)$, which gives us \[
		e^{2nc_{1}}\length(\gamma)\leq	\length(f^{2n}(\gamma))\leq2^{n}\sqrt{1+\rho^{2}}\diam Q.
		\]Consequently, \[
		e^{c_{1}}\leq	\sqrt{2}\left[\length(\gamma)^{-1}\sqrt{1+\rho^{2}}\diam Q\right]^{1/(2n)},
		\]which contradicts condition \ref{stepA2-expansion} for small perturbations.
	\end{proof}
	Let us recall that in the case of non-uniformly hyperbolic H\'enon family, it was proved by Benedicks and Viana that the basin of attraction of the strange attractor found by Benedicks and Carleson \cite{henon-dynamics-of} contains a neighbourhood of the attractor. On the other hand, in the case of piecewise uniform hyperbolicity, we can prove that in a neighbourhood of the attractor, the stable manifold of the fixed point $z_+$ is dense. This fact follows from high uniform expansion and uniform separation of the angles of stable and unstable local manifolds. Moreover, this fact was the key component of the proof that the map in consideration restricted to the closure of the unstable manifold of $z_+$ is mixing; see \cite{simpson:robust-chaos, kucharski:str-att,Lozi-likemaps, strange-attractor-mis}. Following the same line of reasoning as in \cite{simpson:robust-chaos, kucharski:str-att,Lozi-likemaps, strange-attractor-mis} we prove the following. We begin with the lemma showing density of local unstable manifolds, a proof for Lozi family can be found in \cite{strange-attractor-mis}. A proof in the case of admissible families with piecewise uniform hyperbolicity is the same, and so we leave it to the reader.
	\begin{lemma}\label{lemma:density-of-unst-mani}
	    For Lebesgue almost every point $z\in Q$  the unstable manifold $\UnstableLocMan{z}$ exists and is a $u$-curve. In particular, the set $\{z\in Q\colon \UnstableLocMan{z}\text{ exists and is a }u\text{-curve}\}$ is dense in $Q$.
	\end{lemma}
    
    \begin{corollary}\label{corollary:stmani-dense-in-R}
		The stable manifold $\StableMan{z_{+}}$ is dense in $Q$. Moreover, $\StableMan{z_{+}}$ intersects every arc of any unstable manifold in $Q$, and $f|_{\cl\UnstableMan{z_{+}}}$ is mixing.
	\end{corollary}
	\begin{proof}
		Let us sketch a proof. First, note that the piecewise uniform hyperbolicity implies the existence of invariant manifolds almost everywhere in $Q$, for the proof, we refer the reader to \cite{strange-attractor-mis}. Assume $U\subset Q$ is open. Find $z\in U$ with such that $\UnstableLocMan{z}\subset U$ is a $u$-curve by Lemma \ref{lemma:density-of-unst-mani}. Then by Lemma \ref{proposition:st-mani-X-intersects-arcs} $\StableMan{z_{+}}\cap \UnstableLocMan{z}\subset U$ is nonempty, and so $\StableMan{z_{+}}$ is dense in $Q$. We now show mixing. The proof can be found in \cite{kucharski:str-att}, and we enclose it for completeness.
		
		Let $U$, $V$ be open sets in $Q$, intersecting $\UnstableMan{z_{+}}$.
		\begin{claim*}
			There exists $k_{1}\in \N$ such that $f^{-k}(V)$, for $k>k_{1}$, intersects $f(Q)$ in such a way that every $u$-curve in $f(Q)$ intersecting both $\Rc$ and $\Ryc$ also intersects $f^{-k}(V)$.
		\end{claim*}
		\begin{proof}
			Note that every $u$-curve in $f(Q)$ that intersects both $\Rc$ and $\Ryc$, must intersect $\StableLocMan{z_{+}}$ by \ref{lemma:basic-properties:prop4} of Lemma \ref{lemma:basic-properties}. Such a $u$-curve will also intersect any sufficiently long $s$-curve contained in some small neighborhood of $\StableLocMan{z_{+}}$. Hence, it suffices to prove that $f^{-k}(V)$ contains such a $s$-curve for $k>k_{1}$ and $k_{1}\in \N$ large enough. This, on the other hand, follows directly from the density of $\StableMan{z_{+}}$. Indeed, one can find $z\in V\cap \StableMan{z_{+}}$ with $\StableLocMan{z}\subset V$. Then $f^{-k}(\StableLocMan{z})$ contains the desired $s$-curve for $k$ big enough. Indeed, in the neighborhood of $\StableLocMan{z_{+}}$ the map $f$ is hyperbolic, so one uses the inclination lemma (see, for example, \cite{katok_Hasselblatt_1995}) to deduce that there are $s$-curves $\gamma_{k}$ contained in $f^{-k}(\StableLocMan{z})$ converging to $\StableLocMan{z_{+}}$, as $k\to \infty$. 
		\end{proof}
		Since $U\cap \UnstableMan{z_{+}}$ is non-empty, there exists a $u$-curve $\gamma\subset U\cap \UnstableMan{z_{+}}$. Using an argument similar to the one in Lemma \ref{proposition:st-mani-X-intersects-arcs}, there exists $n_{1}\geq 0$ such that some curve of $f^{n_{1}}(\gamma)$ intersects both $\Rc$ and $\Ryc$. Therefore, it also intersects $f^{-k}(V)$ for all $k\geq k_{1}$ by the Claim above. Hence, $f^n(U)\cap V \cap \UnstableMan{z_{+}}\supset f^{n-n_{1}}(f^{n_{1}}(\gamma)\cap f^{n_{1}-n}(V))\cap \UnstableMan{z_{+}}\neq\emptyset$ for all $n\geq n_{1}+k_{1}$.
	\end{proof}
	\subsection{Examples: BCNF}\label{subsection:c1-perturbations}
	In this section, we show that the family of strongly dissipative border-collision normal forms (BCNF for short) parameterized by a surface can be regarded as an admissible family. Thus, the results of this paper hold for strongly dissipative BCNF.
	
	An example of Lozi-like map that can be found in \cite{Lozi-likemaps} is that of a BCNF, and so in this particular case, a Lozi-like family contains some admissible families by Lemma \ref{lemma:bcnf-are-admissible-families}. Interestingly,  Misiurewicz and \v{S}timac provided numerical evidence that the symbolic dynamics of the said example does not appear in the Lozi family.
	
	The family of border collision normal forms is a generalization of the Lozi family, which is given by \[L_{a,b}(x,y)=(1-a|x|+y,bx),\] for $(a,b)\in\plane$. Note that some authors use different equivalent definitions, giving maps conjugate to $L_{(a,b)}$, see \cite{yutaka-ishii1997:kneading1, dyi-shing-ou:critical-points-I}. Naturally, the Lozi family is isomorphic (in the sense of Lemma \ref{lemma:bcnf-are-admissible-families}) to an admissible family for an open set of parameters near $(0,2)$; see Lemma \ref{corollary:lozi-family-admissible}.

	On the other hand, the border-collision normal forms are given by \begin{equation*}
		\bcol_{\xi}\colon(x,y)\mapsto\bcol_{\xi}(x,y):=\left\{ \begin{array}{ll}
			(\tau_{-}x+y+1,\delta_{-}x) & \textrm{ for } (x,y)\in \Rn,\\
			(\tau_{+}x+y+1,\delta_{+}x) & \textrm{ for }(x,y)\in\Rp,
		\end{array} \right.
	\end{equation*}where $\xi=(\tau_{-},\tau_{+},\delta_{-},\delta_{+})\in\R^{4}$ and $\Rn=\{(x,y)\colon x\leq 0\}$, $\tau_{+}<0<\tau_{-}$, $\Rp=\{(x,y)\colon x\geq 0\}$.  The existence of a trapping region, an expanding cone family, and chaotic properties in an open parameter region is shown in \cite{simpson-gosh:robust-chaos,simpson:robust-chaos}. We note that these results treat parameters for which the maps are relatively far from the one dimensional families. In particular, one cannot hope to apply the results of this paper to the whole parameter regions found in \cite{simpson-gosh:robust-chaos,simpson:robust-chaos}. However, many smooth parameterizations $(a,b)\mapsto (\tau_{-}(a),\tau_{+}(a),\delta(b),\delta(b))$ will give us an admissible family. 
	
	In \cite{simpson:robust-chaos} Glendinning and Simpson show that BCNF possesses a chaotic attractor, that is, an attractor on which the map is mixing, for some open parameter region in $\R^{4}$ on which BCNFs are orientation preserving. The authors do not show that the attractor is maximal. In fact, they claim there is numerical evidence for the lack of maximality. Let $\R_{\geq 0}=[0,\infty]$.
	\begin{lemma}\label{lemma:bcnf-are-admissible-families}
		Let \[\xi\colon \R_{\geq 0}^{2}\to\R^{4},\quad\xi(a,b)= (\tau_{-}(a),\tau_{+}(a),\delta(b),\delta(b)),\]  be a $C^{\infty}$ smooth function on a neighbourhood of $\R_{\geq 0}\times\{0\}$ satisfying $b^{-1/k}\delta(b)\to 0$, as $b\to 0$, $b^{-1/k}\delta(b)$ is $C^2$ for $b\geq 0$ and some $k\in\N$, and $\tau_+(a_*)\left(1-\tau_-(a_*)\right)=\tau_-(a_*)$ for some $a_{*}\in\R_{\geq 0}$ with $|\tau_\pm(a_*)|>\sqrt{2}$ and $\tau_{+}(a_*)<0<\tau_{-}(a_*)$. Then there exist an open region $P\subset \R_{\geq 0}^{2}$ with closure intersecting $\R_{\geq 0}\times\{0\}$, and a $C^{\infty}$ diffeomorphism $h\colon \plane\to\plane$, such that $\{h^{-1}\circ B_{\xi(a,b)}\circ h\}_{(a,b)\in P}$ is an admissible family.
	\end{lemma}
	\begin{proof}
		Recall $\tau_{+}<0<\tau_{-}$. Consider $|\tau_-|,|\tau_+|>1$. Let us define $r_{\tau}\colon \R\to\R$ by $r_{\tau}(x):=\tau_{\sigma}x+1$ for $x\sigma>0$ and $\sigma=\pm 1$. Note that $z_{\sigma}:=1/(1-\tau_{\sigma})$, $q_-:=z_{-}<0<z_{+}$ for $\sigma=\pm1$ are fixed points of $r_{\tau}$. Moreover, denoting $q_{+}\in r_{\tau}^{-1}(z_{-})$, with $q_{+}\neq q_{-}$, we have $q_{+}>0$, and so $[q_{-},q_{+}]$ is $r_{\tau}$-invariant and contains the critical point $0$. Note that $q_+=\frac{\tau_-}{\tau_+(1-\tau_-)}$. The condition for $r_{\tau}$ to be surjective is that $1=r_{\tau}(0)=q_+$, that is, $\tau_+(1-\tau_-)=\tau_-$. By assumption, this is satisfied at some $a_\ast$. Consequently, using parameterization $a\mapsto (\tau_{-}(a),\tau_{+}(a))$, map $r_{\tau}$ is a unimodal map from Step A. Put $h(x,y):=(x,b^{1/k}y)$. Denote $g_{a,b}(x,y)=g(x,y)=(g_{1}(x,y),g_{2}(x,y))=(b^{1/k}y,b^{-1/k}\delta(b)x)$. We can then compute for $b>0$\begin{equation*}
			\tilde{B}_{\xi}:=h^{-1}\circ B_{\xi}\circ h(x,y)=(r_{\tau}(x),0)+g_{a,b}(x,y).
		\end{equation*} We see that $\tilde{B}_{\xi}(x,y)$ for $|\tau_{\sigma}|>\sqrt{2}$ satisfies Step A, condition of piecewise uniform hyperbolicity \ref{stepA2-hyp}, and of high expansion \ref{stepA2-expansion}. Finally, note that $g_{a,b}(x,y)|_{K}$ is $C^{2}$ for every compact subset $K\subset \plane$, so it satisfies Step B.
	\end{proof}
	As an immediate consequence of the above, we obtain the following.
	\begin{lemma}\label{corollary:lozi-family-admissible}
		There exists an open region $P\subset \R_{\geq 0}^2$ with $(2,0)\in \cl P$ and a $C^{\infty}$ diffeomorphism $h\colon \plane\to\plane$, such that $\{h^{-1}\circ L_{a,b}\circ h\}_{(a,b)\in P}$ is an admissible family.
	\end{lemma}
	\subsection{Proof of Thm. \ref{mainthm:chaotic-attractors}: existence of chaotic attractors}
	In this section, we show that, under the assumption of piecewise hyperbolicity and high expansion, $\cl\UnstableMan{z_{+}}$ is a chaotic attractor in the sense of Definition \ref{def:chaotic-att}. On the other hand, in later sections, we will tackle the question whether $\cl\UnstableMan{z_{+}}$ is a topological strange attractor, and what the non-wandering set is in $Q$. Recall the statement of Theorem \ref{mainthm:chaotic-attractors}.
	\begin{theorem*}[Thm. \ref{mainthm:chaotic-attractors}]
		Le $\cc{F}:=\{f_{a,b}\}_{(a,b)\in A\times B}$ be a piecewise hyperbolic admissible family, satisfying \ref{stepA2-expansion} and \ref{stepA2-hyp}, and $a_{*}$ be the parameter such that $f_{a_{*}}(c)=q_{+}$, and $f_{a}(c)>q_{+}$ for $a\in (a_{*},a_{2})$. Then there exists $\tilde b\in B$ such that for $b\in (0,\tilde b)$, there are intervals $\tilde A^{b}\subset A$ with $\tilde{A}^{b}\nearrow (\tilde a,a_{*})$, for some $\tilde a\in A$, and such that $\cl\UnstableMan{z_{+}}$ is a chaotic attractor in $\tilde{A}\times\{b\}$. 
	\end{theorem*}
	\begin{proof}[Proof of Thm. \ref{mainthm:chaotic-attractors}]
		The parameter sets $A^{b}$ are fixed so that for $b>0$ small enough, Lemma \ref{lemma:basic-properties} holds. Decreasing $b$ further we may assume that Lemma \ref{lemma:cone-inv} holds. Consequently, Lemma \ref{lemma:position-inv-manifolds-fixed-points} holds, that is, the positions of stable and unstable manifolds of the fixed poits $z_{+}$, $z_{-}$ are as depicted on Figure \ref{fig:manifolds-positions-OR}. These relative positions are necessary for the results of Section \ref{subsection:chaotic-prop} to hold. In particular, for parameters in $A^{b}\times \{b\}$, for $b>0$ small enough, $\StableMan{z_{+}}$ intersects every arc of any unstable manifolds in $Q$ and $f|_{\cl\UnstableMan{z_{+}}}$ is mixing. The former implies that the homoclinic intersections of $\StableMan{z_{+}}$ and $\UnstableMan{z_{+}}$ are dense in $\cl\UnstableMan{z_{+}}$. As the family $\cc{F}$ is piecewise uniformly hyperbolic, those intersections are transverse. Near any transverse homoclinic intersection there is a horseshoe of some renormalization of the original map $f$, see for example \cite{palistakens:hyp}. A horseshoe always contains periodic points. It follows that periodic points are dense in $\cl\UnstableMan{z_{+}}$. The mixing of $f|_{\cl\UnstableMan{z_{+}}}$ implies transitivity. Note that transitivity and density of periodic orbits imply sensitivity to initial conditions by results of \cite{banks-brooks-devaney-chaos-1992}. Overall, we have shown that $f|_{\cl\UnstableMan{z_{+}}}$ has Devaney chaos.
		
		It remains to verify that $\cl\UnstableMan{z_{+}}$ is an attractor. It is a compact, $f$-invariant set. Transitivity of $f|_{\cl\UnstableMan{z_{+}}}$ implies the existence of a dense orbit. The fact that the stable set has nonempty interior can be readily seen by finding a disc $W$ bounded by curves contained in $\UnstableMan{z_{+}}$ and $\StableMan{z_{-}}$, as in the proof of Lemma \ref{lemma:maximality-OR} (orientation reversing case) or Lemma \ref{lemma:fixedpoint-in-omega(c)->maximality}. For example, we may take $W$ to be bounded by $\StableLocMan{z_{+}}$ and a curve $\tau\subset \UnstableMan{z_{+}}$ that connects $z_{+}$ and the first homoclinic intersection of $\StableLocMan{z_{+}}$ and $\UnstableMan{z_{+}}$ in $\Rp$.		 
	\end{proof}
	\section{Renormalization model}\label{section:renormalization}
	\begin{figure}[!ht]
		\centering
		\includegraphics[width=\textwidth]{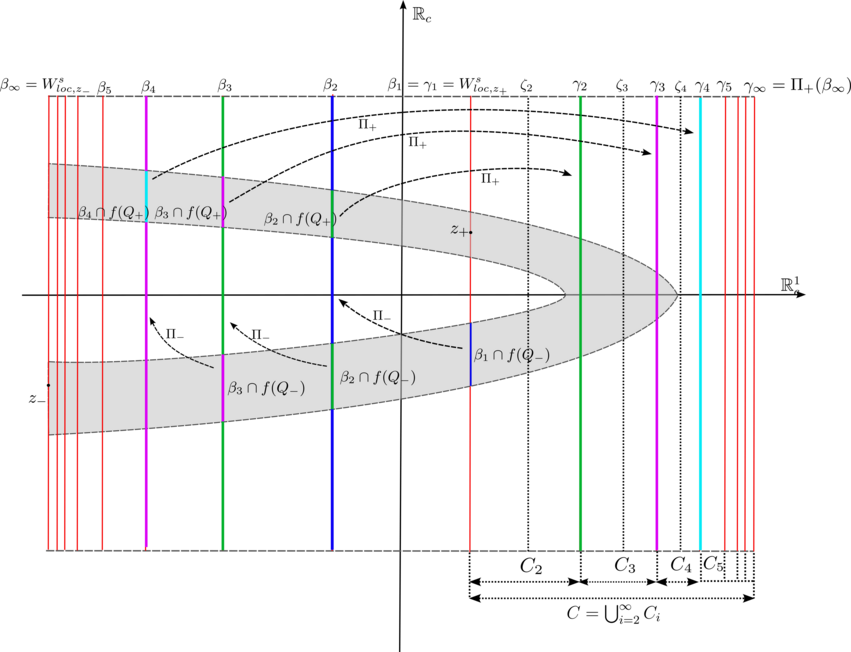}
		\caption{The dynamics of pullbacks of stable curves, creating the renormalization model. The grey region depicts $f(Q)$. We denoted $Q_{\pm}:=Q\cap \bb{R}_{\pm}$.}
		\label{fig:renormalization}
	\end{figure}
	\begin{remark}
		Let us introduce a standing assumption for the rest of this article of $(a,b)\in A^{b}\times\{b\}$, so that Lemmas \ref{lemma:basic-properties} and \ref{lemma:position-inv-manifolds-fixed-points} hold. 
	\end{remark}
	Recall that by $\Rn$ and $\Rp$ we denoted, respectively, the left and right connected components of $\plane\setminus \Rc$. Let $f_{\pm}=f|_{\R_{\pm}}$. For $s$-curves $\gamma_{1},\gamma_{2}$ on $Q$, with disjoint interiors, we have defined $S=S(\gamma_{1},\gamma_{2})$ as the closure of the connected component of $Q\setminus\left(\gamma_{1}\cup\gamma_{2}\right)$ with $\gamma_{1}\cup\gamma_{2}\subset \fr S$. 
	
	We now proceed to a brief reminder of the renormalization construction originally presented in \cite{dyi-shing-ou:critical-points-I} by Dyi-Shing Ou. The construction has been introduced as a way of studying differences between unimodal families and their planar perturbation, the Lozi family, in the order of periodic orbits creations. It was also proposed by Ou as a tool to analyze Hausdorff continuity of strange attractors of the Lozi family \cite{dyi-shing-boronski:disc-lozi-att}.
    
    Note that the construction is valid for the degenerate case of $b=0$. We first define pull-backs of $s$-curves on $Q$. Note that $f(Q)$ is folded along $\Ryc=f(\Rc)$, and $f(Q\cap\Rc)$ is composed of two $u$-curves according to Lemma \ref{lemma:basic-properties}. Let $u_{L}$ be the left utmost point of $f(Q\cap\Rc)$ and $u_{R}$ its right utmost point. Moreover, let $\omega$ be an $s$-curve on $Q$. If $\omega\cap \Ryc $ is to the left of $u_{L}$, then $f^{-1}(\omega)\cap Q$ contains two $s$-curves, one is $\Pi_{-}(\omega)\subset \Rn$ and the other is $\Pi_{+}(\omega)\subset\Rp$, where $\Pi_{\sigma}(\omega):=f_{\sigma}^{-1}(\omega\cap f(\R_{\sigma}\cap Q)) $. In this way, the transformations $\Pi_{-}$ and $\Pi_{+}$ define pull-backs of $s$-curves by the two branches of $f$.
	
	Let us consider sufficiently long stable and unstable local manifolds of $z_-$ and $z_+$ so that Lemma \ref{lemma:basic-properties} holds, that is the manifolds connect upper and lower $u$-curves contained in the boundary $Q$. Recall that $\StableLocMan{z_{+}}\cap Q$ is an $s$-curve on $Q$, connecting lower and upper faces of $Q$. Let $\beta_{1}=\StableLocMan{z_{+}}\cap Q$, and $\beta_{m}=\Pi_{-}(\beta_{m-1})$ for $2\leq m< \infty$, $\beta_{\infty}:=\StableLocMan{z_{-}}\cap Q$, $\gamma_{m}=\Pi_{+}(\beta_{m})$ for $1\leq m <\infty$ and $\gamma_{\infty}=\Pi_{+}(\beta_{\infty})\subset\StableMan{z_{-}}$. The $s$-curves $\{\beta_{m}\}_{1\leq m<\infty}$ are subsets of $\StableMan{z_{+}}$, while, as mentioned before, $\beta_{\infty}\cup\gamma_{\infty}\subset \StableMan{z_{-}}$. Let $C=S(\gamma_{1},\gamma_{\infty})\setminus\gamma_\infty$, $C_{m}=S(\gamma_{m-1},\gamma_{m})$ for $2\leq m < \infty$. The sets $\{C_{m}\setminus \gamma_{i-1}\}_{2\leq m <\infty}$ form a partition of $C\setminus \left(\gamma_1\cup\gamma_\infty\right)$.

	Define the map $h\colon C\to C$ by the formula \[
	h(x) = f^i(x) \text{ for }x \in C_i\setminus \gamma_{i-1},
	\] and $h(x)=f(x)$, for $x\in\gamma_1$. Now, note that $h$ is the first return map of $C$. That is, $h(x) = f^{j(x)}(x)$, where $j(x)=\min\{j\in\N\colon f^j(x)\in C\}$, and $j(x)=i$ on $C_i$. In other words, $C_{i}$ visits $C$ for the first time at the $i$-th  iteration under $f$. Moreover, from the construction we see that $i-1=\min\{j\geq0\colon f^j(C_{i})\cap \Rc\neq\emptyset\}$. Hence, $h|_{C_{i}}=f^{i}|_{C_{i}}$ has a critical curve $\zeta_{i}:=f^{-i+1}(\Rc)\cap C_{i}$. Let us denote by $C^{-}_{i}$, respectively $C^{+}_{i}$, the connected component of $C_{i}\setminus\zeta_{i}$ that is mapped under $f^{i-1}$ to the left half plane and, respectively, the right half plane.
	We can summarize the above discussion in the following manner.
	\begin{prop}\label{proposition:first-return}
		The following holds \begin{enumerate}[label=(\roman*)]
			\item\label{proposition:first-return:c1} The map $h\colon C\to C$ is the first return map to $C$
			\item\label{proposition:first-return:c2} the sets $C_i\setminus \gamma_{i-1}$, $i\geq 2$, are the regions of constant first visit under $f$ to $C$; i. e. $$
			C_i\setminus \gamma_{i-1} = \{z \in \plane\colon \min\{k \in\{1,2,...\}\colon f^k(z)\in C\}=i\}
			$$ 
			\item\label{proposition:first-return:c3} for every $i\geq 2$, there exists an $s$-curve $\zeta_{i}\subset C_{i}$ such that $f^{i-1}(\zeta_i)\subset \Rc$ and $C^{\pm}_i$ are the connected components of $C_i\setminus\zeta_i$
			\item\label{proposition:first-return:c4} $C^{+}_{i}$ is to the left of $C^{-}_{i}$
		\end{enumerate}
	\end{prop}
	\begin{proof}
		By Lemma \ref{lemma:basic-properties} if $(a,b)$ is close to $(a_{*},0)$, the curves $\beta_{i}$ and $\gamma_{i}$, for $i\geq 1$ are well defined $s$-curves. Indeed, for that, it suffices to note that by Lemma \ref{lemma:basic-properties} $\StableLocMan{z_{+}}\cap \Ryc$ is to the left of $u_{L}$, the claim follows by induction on $i$. The fact that $C_i\setminus \gamma_{i-1}$ are regions of constant first visit time follows directly from the construction.
		
		Let us justify \ref{proposition:first-return:c4}. The vertical face $E_+$ of $C^{+}_{i}$ contained in $\StableMan{z_{+}}$, by the construction of the renormalization model, is mapped under $f^{i-1}$ to $\StableLocMan{z_{+}}$, and so necessarily $E_+=\gamma_{i-1}$. Similarly, the vertical face $E_{-}$ of $C^{-}_{i}$ contained in $\StableMan{z_{+}}$ is mapped to $\beta_{2}$ under $f^{i-1}$, and is mapped to $\StableLocMan{z_{+}}$ under $f^{i}$. Hence, $E_{-}=\gamma_{i}$. As $\gamma_{i-1}$ is to the left of $\gamma_{i}$, the same relation holds for $C^{+}_{i}$ and $C^{-}_{i}$.
	\end{proof}
	
	
	Let $0<\lambda:=\mu^{-1}<1$ be the contraction constant from Lemma \ref{lemma:cone-inv}. We refer the reader to Figures \ref{fig:renormalization}, \ref{figure:h(C)-OP} and \ref{figure:h(C)-OR} for geometric clues.
	\begin{prop}\label{proposition:multi-renorm}
	\end{prop}
	We have the following for $i,j\geq 2$ \begin{equation}\label{proposition:multi-renorm:i3:1}
		\pi_{1}(C_{i}\cap \Ryc)\leq \pi_{1}(C_{j}\cap \Ryc)\text{ if and only if } -\lambda^{i}\leq -\lambda^{j},
	\end{equation} and \begin{equation}\label{proposition:multi-renorm:i3:2}
		\pi_{1}(h(C_{i})\cap \Ryc)\leq \pi_{1}(h(C_{j})\cap \Ryc)\text{ if and only if } -(\epsilon\lambda)^{i-1}\leq -(\epsilon\lambda)^{j-1}.
	\end{equation}
	\begin{proof}
The formula \eqref{proposition:multi-renorm:i3:1} follows from the fact that the curves $\gamma_{i}$ satisfy $\pi_{1}(\gamma_{i}\cap \Ryc) \leq  \pi_{1}(\gamma_{i+1}\cap \Ryc)$ and $C_{i}=S(\gamma_{i-1},\gamma_{i})$, for $i\geq 2$.

 Regarding the proof of \eqref{proposition:multi-renorm:i3:2}, first note that $\UnstableLocMan{z_{-}}$ is a $u$-curve contained in the lower half plane, connecting $\beta_\infty=\StableLocMan{z_{-}}$ with $\beta_1=\StableLocMan{z_{+}}$, and so  every $\beta_{i}$, $i\geq 1$, intersects $\UnstableLocMan{z_{-}}$ transversally in the lower half plane. Define inductively
		\begin{align*}
			K_{1}&:=f(C),\\
			K_{m}&:=f(K_{m-1})\setminus f^{m}(C_{m}) \text{ for } m\geq 2.
		\end{align*}

From the construction it follows that $K_{m}$ is a rectagle that connects $\beta_1$ and $\beta_\infty$. Hence, for $i\geq 1$ each $\beta_{i}$ intersects both the upper and lower faces of $K_{m}$ transversally. We define the signed distance between $K_{m}$ and $\UnstableLocMan{z_{-}}$ along $\beta_{i}$ by $\theta_{i,m}:=\tau_{s}\cdot\operatorname{length}(s)$, where $s$ is the shortest $s$-curve contained in $\beta_{i}$ that connects $K_{m}$ and $\UnstableLocMan{z_{-}}$, and 
		$$
		\begin{cases}
			\tau_{s}=1 \text{ if $s$-curve $s$ is above } \UnstableLocMan{z_{-}}\\
			\tau_{s}=-1 \text{ otherwise}.
		\end{cases}
		$$
		Now, it follows from the existence of the stable invariant cone family that 
		\begin{equation}\label{strange-attractors:eq:2}
			|\theta_{i,m}|\leq\lambda|\theta_{i+1,m-1}|
		\end{equation}
		for $i\geq 1$ and $m\geq2$. Moreover, as $f$ reverses for $\epsilon=-1$, and preserves for $\epsilon=1$, the orientation on $\beta_{i}$, we also have \begin{equation}\label{strange-attractors:eq:1}
			\sgn{\theta_{i,m}}=\epsilon\sgn{\theta_{i+1,m-1}}.
		\end{equation} Iterating \eqref{strange-attractors:eq:1} we obtain $$\sgn{\theta_{1,m-1}}=\epsilon^{m-2}\sgn{\theta_{m-1,1}}=\epsilon^{m-2}.$$ The latter equality follows from the fact that $\sgn{\theta_{m-1,1}}=1$. Indeed, note that $\theta_{1,1}$ and $\theta_{2,1}$ are the distances between, respectively, the right and left faces of $f(C_{2})$, and $\UnstableLocMan{z_{-}}$. In general, $\theta_{m-1,1}$ and $\theta_{m,1}$ are the distances between, respectively, the right and left faces of $f(C_{m})$, and $\UnstableLocMan{z_{-}}$. In particular, $\sgn{\theta_{m-1,1}}=1$. Moreover, the right face of $K_{m-1}$ is equal to the right face of $f^{m-1}(C_m)$. It follows that $\sgn{\theta_{1,m-1}}=\epsilon^{m-2}$ is the sign of the distance along $\StableLocMan{z_{+}}$ between the right boundary of $f^{m-1}(C_{m})$ and the point of the intersection $\StableLocMan{z_{+}}\cap \UnstableLocMan{z_{-}}$. As $h$ is a first visit map, the orbits of $C_{m}$ must be disjoint, and so, combining formulas \eqref{strange-attractors:eq:2} and \eqref{strange-attractors:eq:1}, the order on $\{\theta_{1,m-1}\}_{m=2}^{\infty}\subset \R$, as a subset of real numbers, is isomorphic with the order on $\{(\epsilon\lambda)^{m-2}\}_{m=2}^\infty\subset \R$, as a subset of real numbers. Let
		\[E_{m-1}=\StableLocMan{z_{+}}\cap f^{m-1}(C_{m}).\]
		Note that the order on $\{\theta_{i,m}\}$ defines the order on 
		$\cc{E}=\{E_{m-1}\}_{m=2}^{\infty}$
		via the relation $\theta_{1,m-1}\leq \theta_{1,m'-1}$ if and only if $E_{m-1}\leq E_{m'-1}$, for $m,m'\in \{2, 3,...\}$. In other words, $E_{m-1}$ is below $E_{m'-1}$ if and only if $\theta_{1,m-1}\leq \theta_{1,m'-1}$. Consequently, since the orders $\{(\epsilon\lambda)^{m-2}\}_{m=2}^\infty\subset \R$ and $\{\theta_{1,m-1}\}_{m=2}^{\infty}\subset \R$ are isomorphic, $(\epsilon\lambda)^{m-2}\leq (\epsilon\lambda)^{m'-2}$ if and only if $E_{m-1}$ is below $E_{m'-1}$. Now, note that $f$ preserves orientation on $\StableLocMan{z_{+}}$ for $\epsilon=-1$, and reverses otherwise, and so $f(E_{m-1})$ is below $f(E_{m'-1})$ if and only if $-(\epsilon\lambda)^{m-1}\leq -(\epsilon\lambda)^{m'-1}$. Finally, note that the sets $h(C_{m})$ connect $f(E_{m-1})$ with $h(C_{m})\cap\Ryc$, and so the mapping $\pi_{2}(f(E_{m-1}))\mapsto \pi_{1}(h(C_{m})\cap\Ryc)$ preserves the orders on real numbers, this shows \eqref{proposition:multi-renorm:i3:2}.
	\end{proof}
	\begin{figure}[h!]
		\centering
		\includegraphics[width=\textwidth]{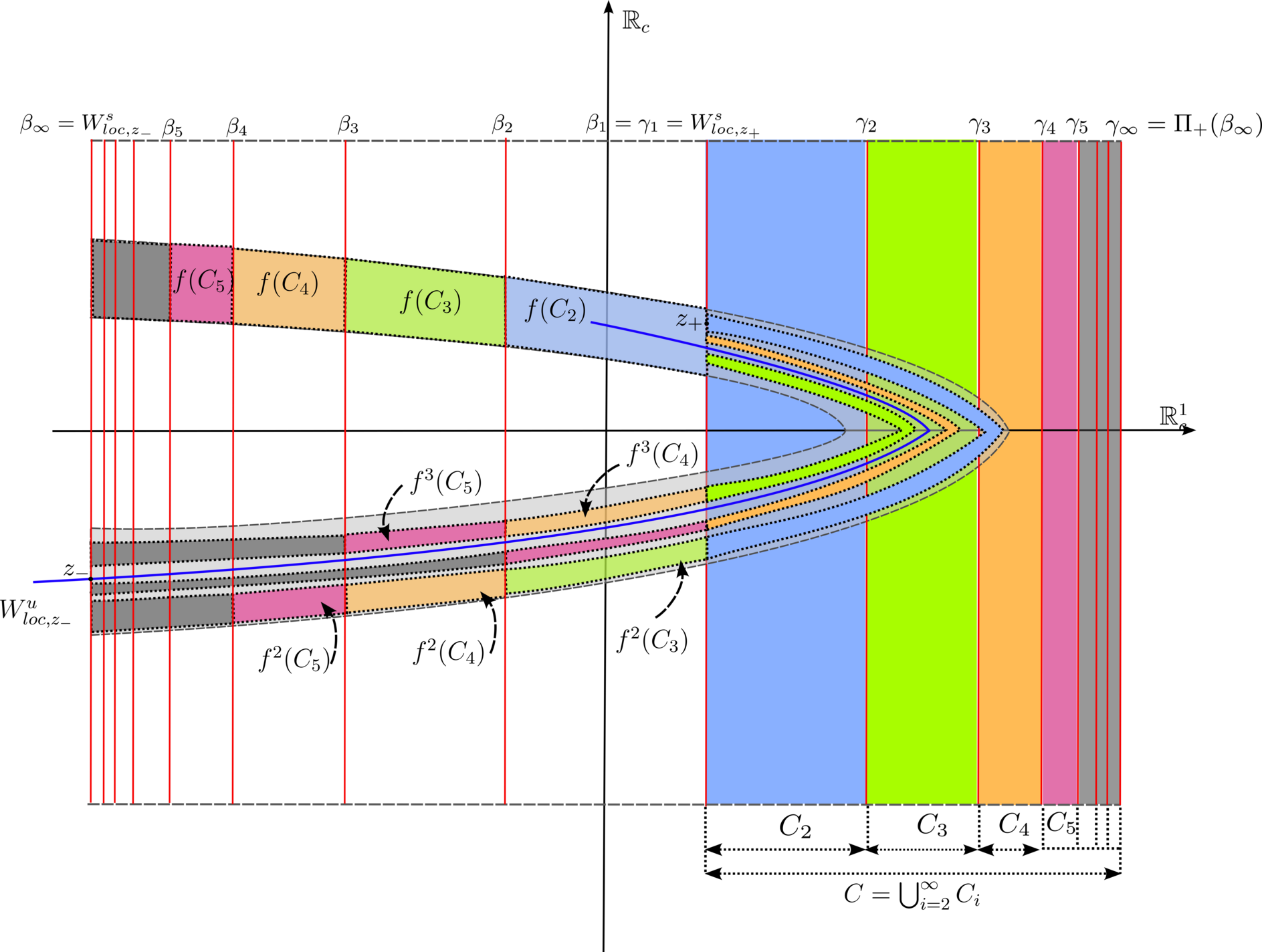}
		\caption{The orbits of domains of first $n$-th return $C_n$ in the orientation reversing case.}
        \label{figure:renormalization-orbits-OR}
	\end{figure}
	\begin{figure}[h!]
		\centering
		\includegraphics[width=\textwidth]{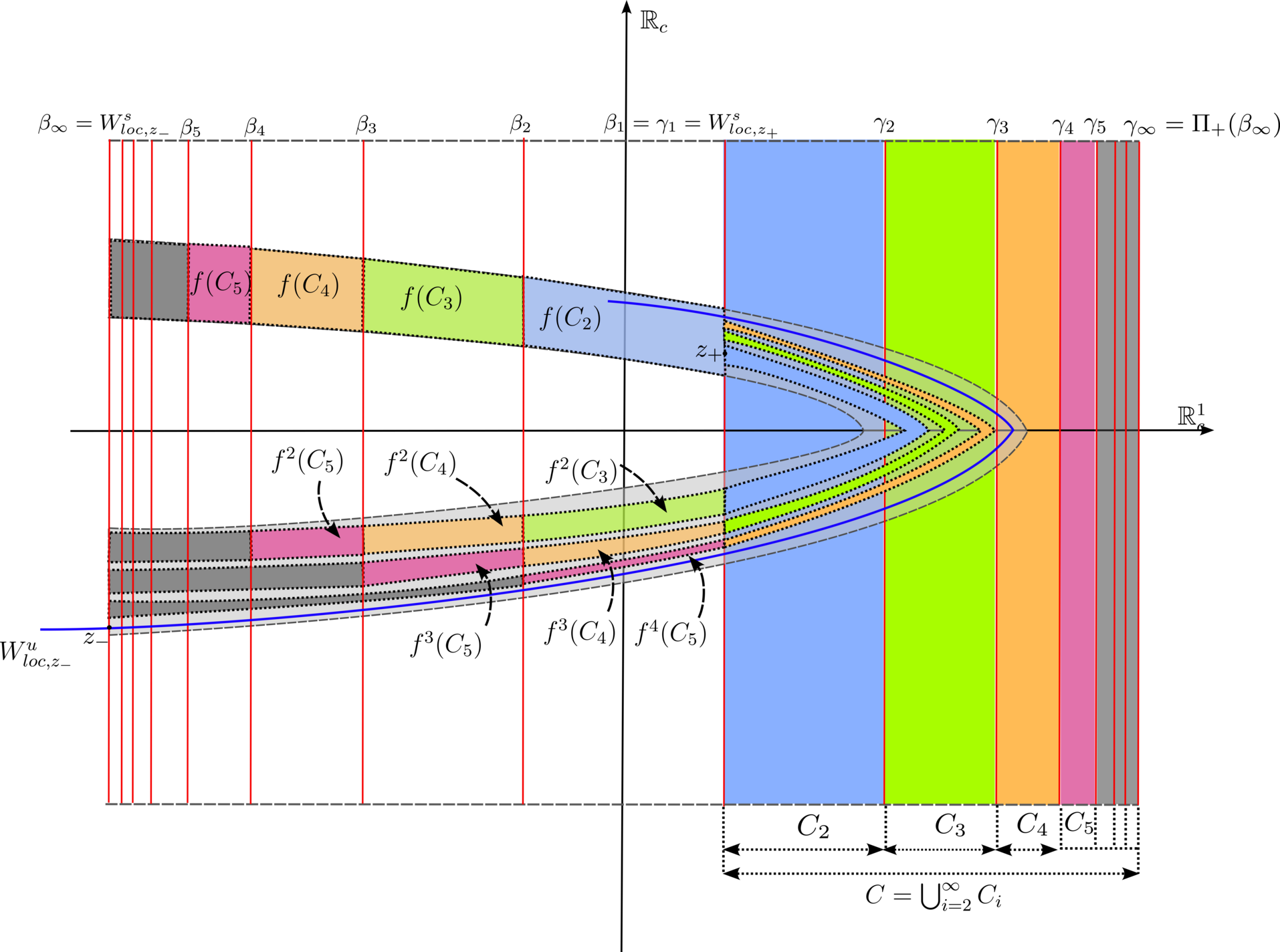}
		\caption{The orbits of domains of first $n$-th return $C_n$ in the orientation preserving case.}
        \label{figure:renormalization-orbits-OP}
	\end{figure}
	\begin{figure}[h!]
		\centering
		\includegraphics[width=\textwidth]{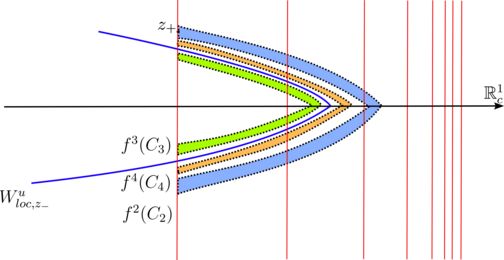}
		\caption{The first return of domains $C_{n}$, the orientation reversing case.}
		\label{figure:renormalization-orbits-OR-C}
	\end{figure}
	\begin{figure}[h!]
		\centering
		\includegraphics[width=\textwidth]{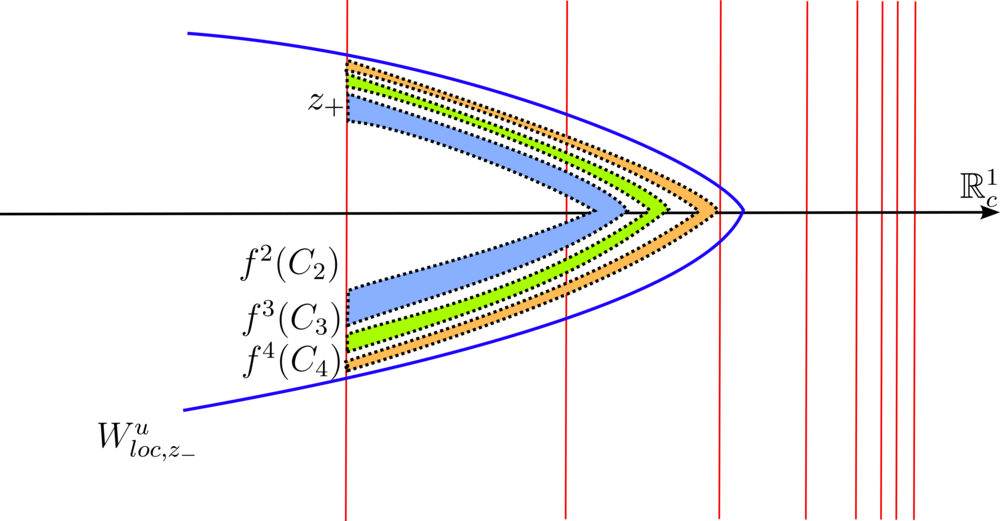}
		\caption{The first return of domains $C_{n}$, the orientation preserving case.}
		\label{figure:renormalization-orbits-OP-C}
	\end{figure}
	\begin{figure}[h!]
		\centering
		\includegraphics[width=\textwidth]{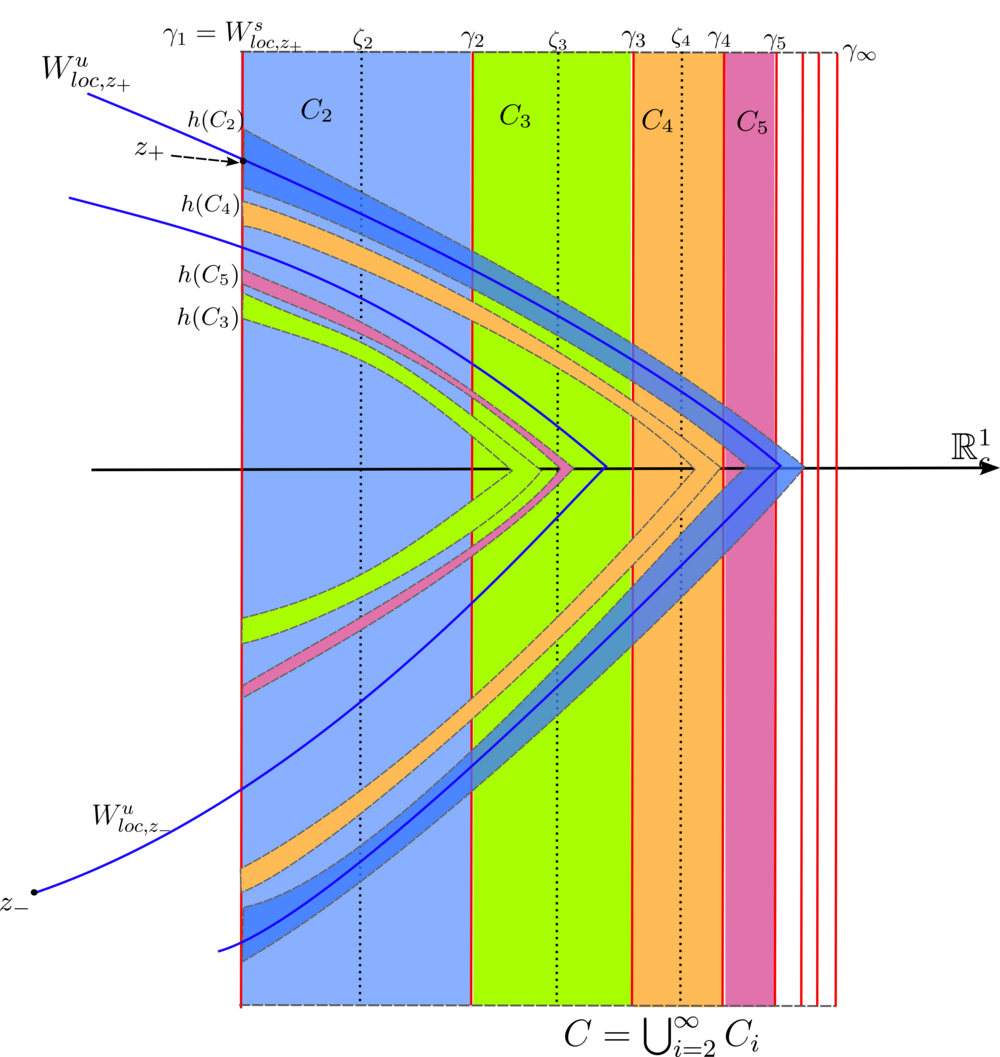}
		\caption{The orientation reversing case: the images $h(C_{i})$ of domains of first return $C_{i}$ accumulate on $f(\UnstableLocMan{z_{-}})$ in an alternating fashion.}
		\label{figure:h(C)-OR}
	\end{figure}
	
	\begin{figure}[h!]
		\centering
		\includegraphics[width=\textwidth]{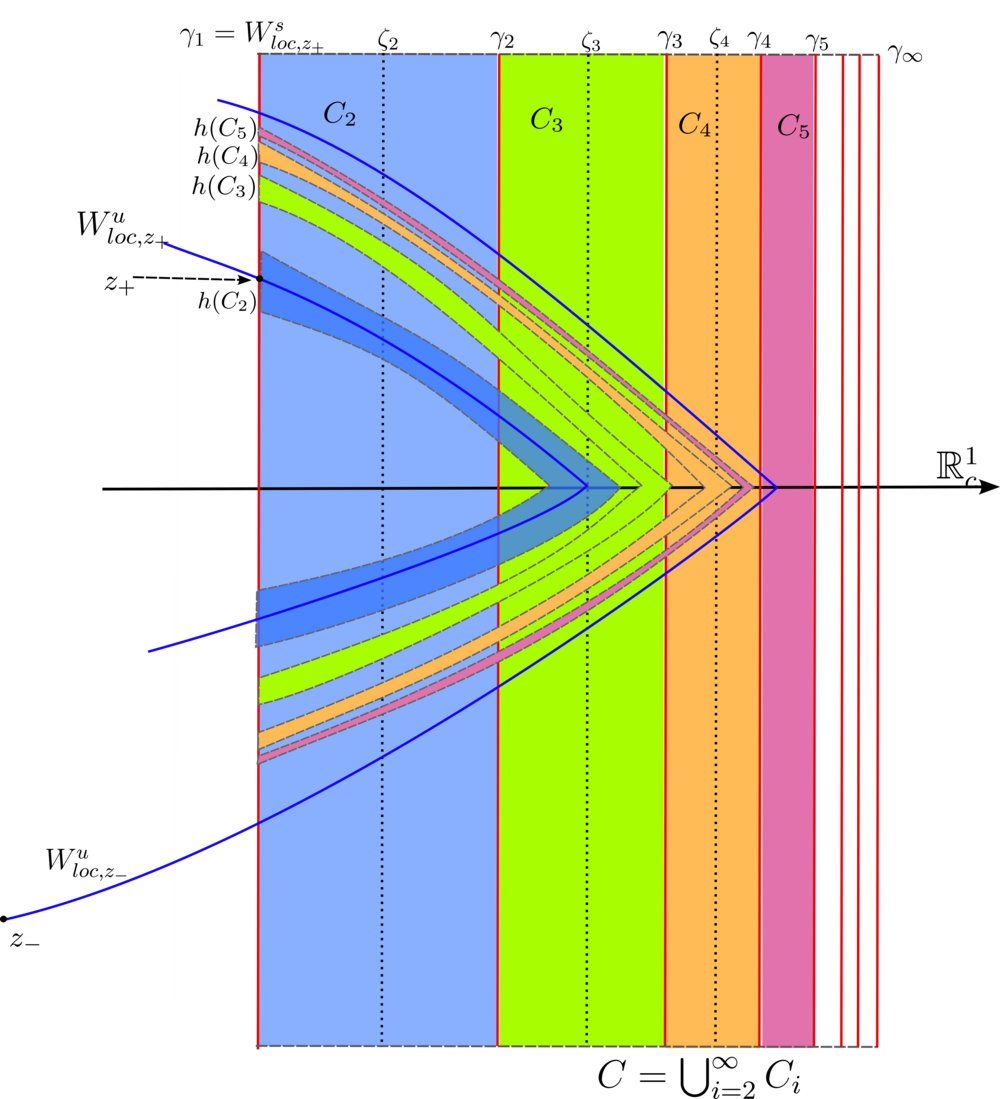}
		\caption{The orientation preserving case: the images $h(C_{i})$ of domains of first return $C_{i}$ accumulate on $f(\UnstableLocMan{z_{-}})$ in an increasing fashion.}
		\label{figure:h(C)-OP}
	\end{figure}
	\section{Maximal attractors}\label{section:maximal-attractors}
	\subsection{Trapping region}
	%
	%
	
	Note that region $Q$ is not a trapping region for $f$, even though it is invariant and $\cl\UnstableMan{z_{+}}\subset Q$ (see Figures \ref{figure:renormalization-orbits-OR} and \ref{figure:renormalization-orbits-OP}). Indeed, it is a rectangle with the left face being $\StableLocMan{z_{-}}$, which is $f$-invariant, and so $\fr Q\cap f(\fr Q)\neq \emptyset$. Nevertheless, it is straightforward to find a trapping region $T$ in $Q$ that contains $\cl\UnstableMan{z_{+}}$. We propose such a region with a rather simple proof once the renormalization model is established. In literature, examples of other trapping regions can be found; see \cite{simpson:robust-chaos,kucharski:str-att, strange-attractor-mis}. The existence of such a trapping region is a step in the proof that the attractor $\Omega:=\bigcap_{i\geq 0}f^{i}(T)$ is maximal. In later sections, we will show that for any trapping region $T\subset Q$, we have $\Omega=\cl\UnstableMan{z_{+}}$. This will imply transitivity on $\Omega$, and will essentially show that $\Omega$ is a topological strange attractor in the sense of Definitions \ref{def:attractor} and \ref{def:strange-attractor}.
	
	We remark here that the assumption of $\cl\UnstableMan{z_{+}}\subset T$ is necessary to show chaotic properties, see Corollary \ref{corollary:stmani-dense-in-R}. That is, the proof of mixing on $\cl\UnstableMan{z_{+}}$ uses the fact that the set of homoclinic intersections is dense in $\cl\UnstableMan{z_{+}}$. If $\Omega$ is strictly greater than $\cl\UnstableMan{z_{+}}$, mixing may fail on $\Omega\setminus\cl\UnstableMan{z_{+}}$. For example, there may be periodic orbits outside $\cl\UnstableMan{z_{+}}$.
	
	\begin{lemma}\label{lemma:trapping-region} For parameters in $A^{b}\times \{b\}$, for $b>0$ sufficiently small, there exists a trapping region $T$ for $f$ with $\cl\UnstableMan{z_{+}}\subset T\subset Q$.
	\end{lemma}
	\begin{proof}
		Let $k:=\min\{i\in\N\colon f(Q)\cap C_{i}=\emptyset \}$. Note that by Lemma \ref{lemma:basic-properties} $k$ is well defined, see Figure \ref{fig:renormalization}. Recall that $\beta_{i'}$ and $\gamma_{i}$, for $i',i=1,...$, are curves on $Q$, and so each pair $\beta_{i'}$, $\gamma_{i}$ defines a strip $S(\beta_{i'},\gamma_i,Q)$, see Definition \ref{definition:strip}. Let $T=S(\beta_{k+1},\gamma_k,Q)$. We now show that $T$ is a trapping region containing $\cl\UnstableMan{z_{+}}$.
		
		Let us denote $B_{i}=S(\beta_{i},\beta_{i+1},Q)$, for $i\geq 1$, and $B_0:=C_2$. Recall $C_{i}=S(\gamma_{i-1},\gamma_{i},Q)$, for $i\geq 2$. First, note that \begin{equation}\label{lemma:trapping-region:eq:2}
			T=\bigcup_{1\leq i\leq k}B_{i}\cup \bigcup_{2\leq j\leq k}C_{j}.
		\end{equation} Moreover, as curves $\gamma_{j}$ are pullbacks of $\beta_{j}$, that is, $\gamma_{j}=f^{-1}(\beta_{j}\cap f(Q\cap \Rp))$, we have \begin{multline}\label{lemma:trapping-region:eq:1}
			f(C_{j})=S\left(\beta_{j-1}\cap f(Q\cap \Rp),\beta_{j}\cap f(Q\cap \Rp),f(Q\cap \Rp)\right)\subset \\\beta_{j-1}\cap f(Q\cap \Rp)\cup \beta_{j}\cap f(Q\cap \Rp)\cup \inter B_{j-1}.
		\end{multline} Note $\beta_{i}\cap f(Q)\subset \inter (B_{i}\cup B_{i-1}) $, for $i=2,...$; see Figure \ref{fig:renormalization}. Consequently, taking \eqref{lemma:trapping-region:eq:1} into account, we have $f(C_{j})\subset \inter (B_{j-2}\cup B_{j-1}\cup B_{j})$. Taking the union over $j=2,...,k$, we obtain \[
		f(\bigcup_{j=2}^{k}C_{j})\subset \inter \bigcup_{i=0}^{k}B_{i}=\inter \left(\bigcup_{i=1}^{k}B_{i}\cup C_2        \right).
		\]From the definition \eqref{lemma:trapping-region:eq:2} of $T$, it follows that, in order to show $f(T)\subset \inter T$, it remains to prove $f(B_{i})\subset \inter T$, for every $i=1,...,k$.
		
		As before, by the definition of curves $\beta_{i}$, if $i>2$, we have $f(B_{i})\subset \inter (B_{i-2}\cup B_{i-1}\cup B_{i})\subset \inter T$. If $i=2$, similarly we have $f(B_{2})\subset \inter (C_{2}\cup B_{1}\cup B_{2})\subset \inter T$. The case of $i=1$, that is, $f(B_{1})\subset \inter T$ follows from the definition of $k$. That is, we see $f(Q)\cap C\subset \bigcup_{r=2}^{k-1}C_{r}$. In fact, by the $f$-invariance of $Q$ and the uniform contraction on $\gamma_{1}=\StableLocMan{z_{+}}$ we have \[
		f(Q)\cap C\subset \inter\gamma_{1}\cup \inter\gamma_{k-1}\cup\inter\bigcup_{r=2}^{k-1}C_{r},
		\]where for a curve $\gamma$ we denote by $\inter\gamma$ the curve without the endpoints. We now note that $f(B_{1})=f(Q)\cap C$ and $\inter\gamma_{1}\subset \inter (C_{2}\cup B_{1})$, $\inter\gamma_{k-1}\subset \inter (C_{k-1}\cup C_{k})$. In consequence, $f(B_{1})\subset \inter ( C_{2}\cup B_{1}\cup C_{k-1}\cup C_{k})\subset \inter T$, which concludes the proof of Lemma \ref{lemma:trapping-region}.
	\end{proof}
	
	\subsection{A geometric condition for maximality}
	This section introduces the conditions that imply that the omega limit set of a given point is contained in $\cl\UnstableMan{z_{+}}$. Later, this will be essential in showing that whenever the critical point $c$ of the one dimensional map $f_{a_{0}}$, with $a_{0}\in A$, is eventually mapped to the fixed $z_{+}$, we have $\Omega=\cl\UnstableMan{z_{+}}$ for some parameters $(a,b)$ close to $(a_{0},0)$. As all orbits of the points in $Q$ visit $C$ at some point, we will restrict our attention to the renormalization map $h\colon C\to  C$. Furthermore, we will deal separately with two kinds of point $z\in Q$, those whose orbits omit $C_{\delta}$, and those whose do not. We will show that if a point $z$ has orbit disjoint with $C_{\delta}$, its local stable manifold $\StableLocMan{z}$ is almost vertical. In the proof of maximality in Theorem \ref{mainthm:maximal-attractors}, it will be used to show that $\StableLocMan{z}\cap\UnstableMan{z_{+}}\neq\emptyset$, implying $\omega(z)\subset\cl \UnstableMan{z_{+}}$. For the second kind, we will show that for an open set of parameters $a\in A$ for which $f_{a}(c)$ is close to the stable set of $z_{+}$, $C_{\delta}$ eventually is mapped under $f_{a,b}$ to a region bounded by curves contained in $\StableMan{z_{-}}$ and $\UnstableMan{z_{+}}$, for any $(a,b)$ close to $(a,0)$ and $\delta$ small enough. The following lemma can be deduced from Theorem 4 in \cite{henon-dynamics-of}.
	
	\begin{lemma}\label{lemma:manifold-bounded-region->max}
		Let $T\subset C$ be a disk with $\fr T\subset \StableMan{z_{+}}\cup\UnstableMan{z_{+}}$. Then for any $P\subset C$ for which there exists $k\in\N$ with $h^{k}(P)\subset T$, we have $\omega_{h}(P)\subset\cl\UnstableMan{z_{+}}$.
	\end{lemma}
	Note that Lemma \ref{lemma:manifold-bounded-region->max} and Proposition \ref{proposition:multi-renorm} immediately implies maximality for $\cl\UnstableMan{z_{+}}$ in the case of $f$ orientation reversing ($\epsilon=-1$). We formulate this result under the assumption of small perturbation, that is, small $b$, since for admissible families, it is a necessary assumption to deduce the existence of a trapping region. However, note that in many examples such as Lozi maps or border collision normal form family, the result holds in a much larger parameter region.
	\begin{lemma}\label{lemma:maximality-OR}
		For any $b$ small enough, there are open sets $\tilde A^{b}\subset A^{b}$ such that for parameters in $\tilde A^{b}\times \{b\}$, $\Omega=\cl\UnstableMan{z_{+}}$, if $f$ is orientation reversing.
	\end{lemma}
	\begin{proof}
Note that it is enough to show $\Omega\cap C=\cl\UnstableMan{z_{+}}\cap C$. We consider the two cases for $z\in C$ and show that $\omega(z)\subset \cl\UnstableMan{z_{+}}$. The first case is when  $h^{i}(z)\in C_{2}$ for every $i\geq 0$, and the second case is when $h^{i}(z)\in C_{j}$, for some $i\geq 0$ and $j\geq 3$. We consider the parameters in a subset $\tilde A^{b}\times\{b\}\subset A^{b}\times\{b\}$ for which we have $\sup\pi_{1}(C_{2}\cap\Ryc)< \inf\pi_{1}(h(C_{2})\cap\Ryc)$, that is, $h|_{C_{2}}\colon C_{2}\to h(C_{2})$ creates a horseshoe. First, assume $h^{i}(z)\in C_{2}$ for every $i\geq 0$. Then $\omega(z)$ is contained in the non-wandering set of the horseshoe $h|_{C_{2}}$. Note that $C_{2}\cap h(C_{2})$ is a union of two rectangles. Each rectangle has its stable faces connected by an arc of $\UnstableMan{z_{+}}$. Consequently, the non-wandering set of $h|_{C_{2}}$ is contained in $\cl \UnstableMan{z_{+}}$, and so is also $\omega(z)\in \cl \UnstableMan{z_{+}}$.
		
		Let us now consider the case of $h^{i}(z)\in C_{j}$, for some $i\geq 0$ and $j\geq 3$. By Proposition \ref{proposition:multi-renorm} we have $\pi_{1}(h(C_{j})\cap \Ryc)\leq \pi_{1}(h(C_{2})\cap \Ryc)$. Moreover, the images of stable faces of $C_{2}$, that is, the connected components of $\StableLocMan{z_{+}}\cap h(C_{2})$ are connected by an arc $\tau$ contained in $\UnstableMan{z_{+}}\cap h(C_2)$, see Figure \ref{figure:h(C)-OR}. It follows that $h(C_{j})$ is contained in a region $\tilde C$ bounded by $\StableLocMan{z_{+}}=\gamma_{1}$ and $\tau\subset \UnstableMan{z_{+}}$. Hence, by Lemma \ref{lemma:manifold-bounded-region->max}, we have $\omega(z)\subset \UnstableMan{z_{+}}$, which concludes the proof of Lemma \ref{lemma:maximality-OR}.
	\end{proof}
	In the remainder of the section we will deal with the case of $f$ orientation preserving.
	\begin{lemma}\label{lemma:z-omits-cdelta->almost-vert-stmani}
		Let $a\in A^{b}$, for some $b\in B$. Assume that the family $\{f_{a,b}\}$ is orientation preserving and for some $\delta>0$ and $j\geq 3$ we have $f(C_{\delta})\subset C_{j}$. Then if $z\in \bigcap _{j\geq 0} h^{j}(C)$ satisfies $h^{k}(z)\notin C_{j}$ for $k\in\Z$, there exists a stable manifold $\StableLocMan{h^{-r}(z)}$ that is almost vertical for some $r\in \N_+$ and is an $s$-curve connecting the upper and lower faces of $C$.
	\end{lemma}
	\begin{proof}
		Let $C^{(j)}=\bigcup_{i< j}C_{i}$. We claim that $H_{\ell}:=h^{\ell}(C^{(j)})\setminus C_{j}$ is a disjoint union of rectangles, each rectangle has the left vertical face contained in $\gamma_{1}=\StableLocMan{z_{+}}$ and the right vertical face contained in $\gamma_{j-1}$, which is the left vertical face of $C_{j}$. Note that by $f(C_{\delta})\subset C_{j}$, the claim is true for $\ell=1$. Now, assume the claim for $\ell>1$. Let $\tilde H$ be a connected component of $H_{\ell}$. By assumption, $\tilde H\cap C_{i}$, $i<j$,  is a rectangle with the left face contained in $\gamma_{i-1}$ and the right face in $\gamma_{i}$. Therefore, the orbit segment of length $i$ of every $\tilde H\cap C_{i}$ under $f$ crosses the singularity region $\Rc$ exactly once, and the $f^{i-1}$ images of vertical faces of $\tilde H\cap C_{i}$ intersect $\Rc$ transversally. Hence, $h(\tilde H\cap C_{i})$ is folded along $\Ryc$, and $h(\tilde H\cap C_{i})\cap \Ryc\subset C_{j}$ and so $h(\tilde H\cap C_{i})\setminus C_{j}$ has two connected components, which are rectangles as described before, as required by the claim.
		
		Let $z\in \bigcap _{j\geq 0} h^{j}(C)$ satisfy $h^{k}(z)\notin C_{j}$ for $k\in\Z$. As $f(C_{\delta})\subset C_{j}$, we have $f^{k}(z)\notin C_{\delta}$. By Proposition \ref{proposition:hyperbolic-set-admissible-family}, $z$ has an almost vertical stable manifold. Moreover, $z\in H_{r}$, for every $r\geq 1$. Therefore, there are connected components $\tilde H_{r}$ of $H_{r}$ such that $z\in \tilde H_{r}$. By $\det \diff h <1$, we have $\Leb(\tilde H_{r})\to 0$, as $r\to \infty$, and so there exists $\StableLocMan{z}$ which connects the upper and lower faces of $\tilde H_{r}$, for some $r>1$. The horizontal faces of $\tilde H_{r}$ are images of subcurves of the horizontal faces of $C^{(j)}$. It follows that $h^{-r}(\StableLocMan{z})$ necessarily intersects the upper and lower faces of $C^{(j)}$, and so $h^{-r}(\StableLocMan{z})$ is an $s$-curve on $C\supset C^{(j)}$. By the invariance of stable manifolds, $h^{-r}(\StableLocMan{z})$ is a stable local manifold of $h^{-r}(z)$.
	\end{proof}
	
	\begin{lemma}\label{lemma:fixedpoint-in-omega(c)->maximality}
		Let $a\in A^{b}$, for some $b\in B$. Assume that the family $\{f_{a,b}\}$ is orientation preserving, $f_{a}(c)\in\StableMan{z_{+}}$ and $f_{a,0}(c)\in C_{i}$ for some $i>3$. Then if $(a,b)$ is sufficiently close to $(a,0)$ and $\delta>0$ is small enough, for all $z\in C_{\delta}$ we have $\omega(z)\subset \cl\UnstableMan{z_{+}}$.
	\end{lemma}
	\begin{proof}
		Consider first the renormalization model in the case of $b=0$. By assumption, for arbitrary $\epsilon>0$ we can find $k\in\N$ such that $f_{a_{0},0}^{k}(c)\in B_{\epsilon}(z_{+})$. Therefore, by perturbation, for $(a,b)$ close enough to $(a_{0},0)$, for $\delta>0$ small enough, we have $U:=f_{a,b}^{k}(C_{\delta})\subset B_{\epsilon}(z_{+})$. Now, we separately consider the subset of $U$ that is to the right of $\StableLocMan{z_{+}}$, that is, $U_{+}:=U\cap C_{2}$, and the subset that is to the left of $\StableLocMan{z_{+}}$, that is, $U_{-}:=U\setminus C_{2}$.
		
		In the case of $U_{-}$, if $\epsilon$ is small enough, since $f_{a,b}|_{\UnstableLocMan{z_{+}}}$ is orientation reversing by Lemma \ref{lemma:position-inv-manifolds-fixed-points}, $f_{a,b}(U_{-})$ is to the right of $\StableLocMan{z_{+}}$. Hence, it is enough to show $\omega(z)\subset \cl\UnstableMan{z_{+}}$ for points to the right of $\StableLocMan{z_{+}}$, that is $z\in U_{+}$. If $\epsilon$ is small enough, $U_{+}\subset C_{2}$. As a consequence, by Lemma \ref{lemma:manifold-bounded-region->max}, we will prove Lemma \ref{lemma:fixedpoint-in-omega(c)->maximality}, once we show $h(C_{2})\subset W$ for some $W\subset Q$ with $\fr W\subset \StableMan{z_{+}}\cup \UnstableMan{z_{+}}$.
		
		For $\delta>0$ and  $b>0$ small enough we have $f(C_{\delta})\subset C_{i}$ with $i\geq 4$. In other words, as $h(C)\cap \Ryc\subset f(C_{\delta})\cap\Ryc$, we have \begin{equation}\label{lemma:fixedpoint-in-omega(c)->maximality:eq:1}
		    \sup\pi_{1}(C_{3}\cap\Ryc)\leq\inf\pi_{1}(C_{i}\cap\Ryc)\leq \inf\pi_{1}(f(C_{\delta})\cap \Ryc)\leq \inf\pi_{1}(h(C)\cap \Ryc).
		\end{equation}Moreover, $\UnstableLocMan{z_+}$ connects $\gamma_1=\StableLocMan{z_+}$ with $h(C)\cap\Ryc$. By \eqref{lemma:fixedpoint-in-omega(c)->maximality:eq:1} this means that $\UnstableLocMan{z_+}$ intersects transversally the vertical faces of $C_3$, and, in fact, any $s$-curve on $C_3$. It follows that we can find a curve $\tau\subset \UnstableLocMan{z_{+}}$ connecting the vertical faces of $C_{3}$. In other words, $\tau$ connects the curves $\gamma_{2}$ and $\gamma_{3}$ (see Figure \ref{figure:h(C)-OP} for the positions of the curves) of $\StableMan{z_{+}}$ that are vertical faces of $C_{3}$. Then $h(\tau)$ connects the connected components of $h(C_{3})\cap\StableLocMan{z_{+}}$, and so, together with $\StableLocMan{z_{+}}$, bound a disc $W$. By Proposition \ref{proposition:multi-renorm}, $\pi_{1}(h(C_{2})\cap \Ryc)\leq  \pi_{1}(h(C_{3})\cap \Ryc)$, and so $h(C_{2})\subset W$. As $\fr W\subset \tau\cup \StableLocMan{z_{+}}$ and $\tau\subset \UnstableMan{z_{+}}$, by Lemma \ref{lemma:manifold-bounded-region->max}, this concludes the proof.
	\end{proof}
	\subsection{Hausdorff continuity for maximal attractors}\label{subsection:hasudorff-cont}
	In this section, we show that the attractors found in the previous section vary continuously in the Hausdorff distance provided the attractor is maximal. The argument appears to be similar as in \cite{barge-cont}, although Barge uses somewhat enigmatic semi-continuity argument. Moreover, the reasoning presented in \cite{kucharski:str-att} seems insufficient. As the Lozi family is an admissible family, the following proofs provide a complete proof of Hausdorff continuity in the case of the orientation preserving Lozi family as well. The case of orientation reversing Lozi family was treated by Glendinning and Simpson in \cite{glendinning2019robust}. Assume $P\subset \R^{2}$ is a compact region of parameters. 
	For a set $H\subset \plane$ and $\epsilon>0$, let $B_{\epsilon}(H)=\bigcup_{h\in H}\{z\in \plane\colon d(z,h)< \epsilon\}$.
	\begin{definition}
		We define the Hausdorff distance of $H,L\subset\plane$ by
		\[ d_{H}(H,L)=\max\{d_{a}(H,L),d_{a}(L,H)\}, \]where $d_{a}(H,L)=\sup_{h\in H}\inf_{l\in L}d(h,l)$.
	\end{definition}
	Let us recall that we equivalently have, for any $H,L\subset \plane$, 		\begin{align}
		\label{subsection:hasudorff-cont:eq:1}d_{H}(H,L)&=\inf\{\epsilon>0\colon H\subset B_{\epsilon}(L)\text{ and } L\subset B_{\epsilon}(H)\}.
	\end{align}

	For a family of piecewise smooth curves $\{\gamma_{p}\colon I=[0,1]\to \plane\}_{p\in P}$, let $\gamma_{p,i}\colon I
    \to \plane$, $i=0,...,k$, for some $k\in\N$ be families of pieces of smoothness. We assume $k$ does not vary in $P$. For a set $B\subset \plane$, let $B^{\complement}$ be its complement in $\plane$.
	\begin{lemma}\label{lemma:curves-cont}
		Let $\{\gamma_{p}\colon I=[0,1]\to \plane\}_{p\in P}$ be a family of continuous curves such that $(p,t)\mapsto \gamma_{p}(t)$ is continuous. Then $p\mapsto B_{\epsilon}(\gamma_{p}(I))^{\complement}$ is continuous in the Hausdorff distance for any $\epsilon>0$.
	\end{lemma}
	\begin{proof}
        Fix $\epsilon>0$ and $R>\epsilon$. Let $\psi(p,t,\alpha,r)$ be the point in $\plane$ that has angular coordinates $(\alpha,r)\in [0,2\pi]\times [0,\infty]$ centered at $\gamma_p(t)$. Then \begin{multline*}
            \cc{B}_{p,R}:= B_{\epsilon}(\gamma_{p}(I))^{\complement}\cap \cl B_R(0)=\\\{\psi(p,t,\alpha,r)\in\plane\colon \epsilon\leq r,~\alpha\in[0,2\pi],~t\in I,~p\in P, ~||\psi(p,t,\alpha,r)||\leq R\}.   
        \end{multline*}
                 
        Since $P\times I\times [0,2\pi]\times [\epsilon,R]$ is compact, $\psi$ is uniformly continuous in $p\in P$. This means that for any $\varepsilon>0$ and a sequence $p_k$, $k\in\N$, convergent to $p'$, we can find $K\in\N$ such that $d(\psi(p_k,t,\alpha,r),\psi(p',t,\alpha,r))<\varepsilon$ for every $k\geq K$, uniformly in other parameters. Consequently, $\cc{B}_{p_k,R}\to\cc{B}_{p',R}$ in Hausdorff distance. Moreover, we have $ B_{\epsilon}(\gamma_{p}(I)) \subset  B_R(0)$ for $R$ big enough. Therefore, \[
        B_{\epsilon}(\gamma_{p}(I))^{\complement}=\cc{B}_{p,R}\cup B_R(0)^{\complement}\]Now, since $B_R(0)$ does not depend on $p$, and $\cc{B}_{p,R}$ is continuous in Hausdorff distance, necessarily $B_{\epsilon}(\gamma_{p}(I))^{\complement}$ is continuous in Hausdorff distance.
	\end{proof}
	\begin{lemma}\label{lemma:three-seq-hausdorff-cont}
		Let $\{A_{p}\}_{p\in P}$ be a family of compact sets, $ F^{n}_{p} $, $ n\in\N $, be a sequence of compact sets such that for every $n\in\N$, the map $p\mapsto F^{n}_{p}$ is continuous in the Hausdorff metric, and $ G^{n}_{p}$, $n\in\N$ be a sequence of families of images of piecewise smooth curves, $G^{n}_{p}=\gamma_{n,p}(I)$, for some $\gamma_{n,p}\colon I\to \plane$, with $p\mapsto G^{n}_{p}$ continuous in the Hausdorff distance for every $n\in\N$. Assume that, for any $p\in P$, we have 
		\begin{equation}\label{eq:hausdorff-cont}
			d_{H}(G^{n}_{p},A_{p})\to 0,\quad d_{H}(F^{n}_{p},A_{p})\to 0,\text{ and  }G^{n}_{p}\subset A_{p}\subset F^{n}_{p},
		\end{equation} as $n\to\infty$, then $p\mapsto A_{p}$ is continuous in the Hausdorff metric.
	\end{lemma}\label{lemma:haus-cont}
	\begin{proof}
		Let $\{p_{k}\}_{k\in\N}$ be a sequence with $p_{k}\to p'$ as $k\to\infty$. Take $ \epsilon>0 $. For large enough $n$ we have $F^{n}_{p'}\subset B_{\epsilon/2}(A_{p'})$. From the continuity of $F^{n}_{p}$ we obtain $A_{p_{k}}\subset F^{n}_{p_{k}}\subset B_{\epsilon/2}(F^{n}_{p'})\subset B_{\epsilon}(A_{p'})$ for large $k$.
		
		On the other hand, by \eqref{eq:hausdorff-cont} $ A_{p'}\subset B_{\epsilon/8}(G^{n}_{p'}) $ for $n$ large enough. From the continuity of $G^{n}_{p}$ in $p$ we get \begin{equation}\label{lemma:three-seq-hausdorff-cont:eq:1}
			B_{\epsilon}(G^{n}_{p_{k}})^{\complement}\subset  B_{\epsilon/2}(G^{n}_{p_{k}})^{\complement}\subset B_{\epsilon/4}( B_{\epsilon/2}(G^{n}_{p'})^{\complement})\subset B_{\epsilon/8}(G^{n}_{p'})^{\complement}\subset A_{p'}^{\complement} 
		\end{equation}for sufficiently big $k$. The second inequality in \eqref{lemma:three-seq-hausdorff-cont:eq:1} follows from the continuity of $p\mapsto B_{\epsilon/2}(G^{n}_{p})^{\complement}$, which is a consequence of Lemma \ref{lemma:curves-cont}. The third one follows from the triangle inequality. Consequently, $ B_{\epsilon}(A_{p_{k}})^{\complement}\subset B_{\epsilon}(G^{n}_{p_{k}})^{\complement}\subset A_{p'}^{\complement} $. Finally, $A_{p'}\subset  B_{\epsilon}(A_{p_{k}})$, proving the lemma.
	\end{proof}
	We now consider the case of the continuity of invariant sets in admissible families. Let $\cc{A}_{p}=\bigcap_{i\geq 0}f_{p}^{i}(T_{p})$, for some trapping region $T_{p}$ that varies continuously in the Hausdorff metric.
	\begin{prop}\label{prop:continuity-hausdorff}
		Let $\cc{A}_{p}$ be an invariant set of $f=f_p$ from an admissible family. Assume that $\cc{A}_{p}=\cl\UnstableMan{z_{+}}$ in some closed subset $P\subset A\times B$. Then $\cc{A}_{p}$ varies continuously in $P$.
	\end{prop}
	\begin{proof}
		Put $A_{p}=\cc{A}_{p}$, $G^{n}_{p}=f^{n}(\UnstableLocMan{z_{+}})$ and $ F^{n}_{p}=f^{n}(T_{p}) $. We claim that the sequences satisfy the assumptions of Lemma \ref{lemma:haus-cont}, and this will show the proposition.
		
		Note that, as $F^{n}_{p}$, $n\in\N$, is a descending sequence of compact subsets and $A_{p}=\bigcap_{n\in\N} F^{n}_{p}$, we necessarily have $d_{H}(F_{p}^{n},A_{p})\to 0$. Indeed, for arbitrary $\epsilon>0$, from the compactness we have $F_{p}^{n}\subset B_{\epsilon}(A_{p})$ for $n$ big enough. 
		
		For the case of $G_{n}=G^{n}_{p}$, let us compactify $\plane$ to $\sphere{2}$ by one point $\infty$ in infinity. Thus, from now on, we consider all subsets of $\plane$ which are unbounded, as containing $\infty$. Similarly as in the case of $F^{n}_{p}$, since $ G_{n}$ is an ascending sequence and $G=\bigcup_{n\in\N}G_{n}$,  $G=\UnstableMan{z_+}$, we have that for $\epsilon>0$, $B_{\epsilon}( G_{n})^{\complement}$ is a descending sequence of compact subsets in $\R^2\cup\{\infty\}$, and $ B_{\epsilon}( G)^{\complement}=\bigcap _{n\in\N} B_{\epsilon}( G_{n})^{\complement}$. Then, as before, we have $ B_{\epsilon}( G_{n})^{\complement}\subset B_{\epsilon/2}( B_{\epsilon}( G)^{\complement})$ for $n$ large enough. Consequently, $ G\subset B_{\epsilon/2}(B_{\epsilon}( G)^{\complement})^{\complement}\subset B_{\epsilon}( G_{n})$, and so also $A_p=\cl G\subset B_{\epsilon}( G_{n})$, proving the claim.
	\end{proof}
	\subsection{Proof of Thm. \ref{mainthm:maximal-attractors}: topological strange attractors}
	In this section, we strengthen the results of Theorem \ref{mainthm:chaotic-attractors}. Recall that for a forward invariant $W\subset \plane$ we have $\omega(W)\subset \cl(\bigcup_{z\in W}\omega(z))$ and $\bigcap_{n\in\N}f^{n}(W)=\omega(W)$. Hence, to show maximality, it is enough to show $\omega(z)\subset \cl\UnstableMan{z_{+}}$ for all $z\in \cl\UnstableMan{z_{+}}$. Let us recall Thm. \ref{mainthm:maximal-attractors}. 
	\begin{theorem*}[Thm. \ref{mainthm:maximal-attractors}]
		Let $\cc{F}:=\{f_{a,b}\}_{(a,b)\in A\times B}$ be an admissible family and $a_{*}$ be the parameter such that $f_{a_{*}}(c)=q_{+}$, and $f_{a}(c)>q_{+}$ for $a\in (a_{*},a_{2})$. Then there are open sets $K_{n}\subset A$, with $K_{n}\to \{a_{*}\}$, as $n\to\infty$, in the Hausdorff metric, and parameters $k_{n}\in B$, with $k_{n}\to 0$, as $n\to\infty$, such that the $\cl\UnstableMan{z_{+}}$ is a maximal invariant set in some neighbourhood for parameters in $K_{n}\times (0,k_{n})$. In particular, at those parameters, under the additional assumption of piecewise uniform hyperbolicity \ref{stepA2-hyp} and high expansion \ref{stepA2-expansion}, $\cl\UnstableMan{z_{+}}$ is a topological strange attractor and varies continuously in the Hausdorff metric. 
	\end{theorem*}
	\begin{proof}
		Note that the renormalization model can be constructed for the parameters in $(\tilde a,a_{*})\times\{0\}$, for some $\tilde a\in A$. That is, one can consider the map $f_{a,0}(x,y)=(f_{a}(x),0)$, and construct the pullbacks of stable curves as in Section \ref{section:renormalization}. In this case, the stable curves are vertical lines, and every point $z$ is projected onto the horizontal axis. Let us denote by $\gamma_{k}=\gamma_{k}^{(a,b)}$, $b\geq0$, $k\geq 1$, the stable curves constructed as pullbacks of $\StableLocMan{z_{+}}$. As $(z,a,b)\mapsto f_{a,b}(z)$ is $C^{1}$ outside $C_\delta$, the manifold $\StableLocMan{z_{+}}$ depends continuously on $(a,b)$, see \cite[Prop. 6.2.21]{katok_Hasselblatt_1995} and Proposition \ref{proposition:inv-manifolds-endo2}, and so also $\gamma_{k}^{(a,b)}$ depend continuously on $(a,b)$. Let $a_{n}\in (\tilde a,a_{*})$, $n\geq 3$, be the parameter such that $f_{a_{n},0}(c)\in \gamma_{n}^{(a_{n},0)}\subset C_{n+1}$. Then we have $f_{a_{n},0}(c)\in C_{n'}$, with $n'=n+1>3$, and so the assumptions of Lemma \ref{lemma:fixedpoint-in-omega(c)->maximality} are satisfied. Consider sufficiently small open intervals $K_{n}\subset (\tilde a,a_{*})$, which are sufficiently close to $a_{n}$, so that for parameters $(a,b)\in K_{n}\times (0,k_{n})$, for some small $k_{n}>0$, we have \begin{equation}\label{proofmainthmB:eq:1}
			f_{(a,b)}(C_{\delta})\subset C_{n+1}\text{ and }h(f_{(a,b)}(Q)\cap C_{n+1})\subset C_{2},
		\end{equation} for some $\delta>0$, and the conclusions of Lemmas \ref{lemma:fixedpoint-in-omega(c)->maximality} and \ref{lemma:z-omits-cdelta->almost-vert-stmani} are satisfied. The latter of \eqref{proofmainthmB:eq:1} can be ensured, as $f_{(a,b)}(Q)\cap C_{n+1}$ is close to $\gamma_{n}$ for $(a,b)$ sufficiently close to $(a_{n},0)$ and $\delta$ small enough, and $h(\gamma_{n})\subset \gamma_{1}\subset C_{2}$.		
		
		We now show that for every $z\in Q$ we have $\omega(z)\subset \cl\UnstableMan{z_{+}}$. Note that every point $z\in Q$ visits $C$, this follows from the renormalization model. Hence, it is enough to show $\omega(z)\subset \cl\UnstableMan{z_{+}}$ for $z\in C$. 
		
		If $h^{i}(z)\in f(C_{\delta})$, for some $\delta>0$ and $i\geq 0$, then the claim follows from Lemma \ref{lemma:fixedpoint-in-omega(c)->maximality}. Therefore, assume $h^{i}(z)\notin f(C_{\delta})$ for any $i\geq 0$. Let $z'\in \omega(z)$. Then $z'\in \bigcap _{j\geq 0} h^{j}(C)$. If $h^{i}(z')\notin C_{n+1}$, for any $i\in \Z$ then, according to Lemma \ref{lemma:z-omits-cdelta->almost-vert-stmani}, the local stable manifold $\StableLocMan{\tilde z}$ of $\tilde z=h^{-r}(z')$, for some $r>1$, is almost vertical and connects the upper and lower faces of $C$. Note that $\UnstableLocMan{z_{+}}$ connects $z_{+}$ with $f(C_{\delta})$, and so $\StableLocMan{\tilde z}\cap \UnstableLocMan{z_{+}}\neq\emptyset$, implying $\omega(\tilde z)\in \cl\UnstableMan{z_{+}}$. Consequently, $\omega(z')=\omega(z)\subset \cl\UnstableMan{z_{+}}$.
		
		On the other hand, assume $h^{i}(z')\in C_{n+1}$ for some $i\in \Z$. Then, by \eqref{proofmainthmB:eq:1} we have \[h^{i+1}(z')\in h(h(C)\cap C_{n+1})\subset h(f(Q)\cap C_{n+1})\subset C_{2}.\]Therefore, it remains to show that for any $z\in C_{2}$ we have $\omega(z)\subset \cl\UnstableMan{z_{+}}$.
		
		Let $z\in C_{2}$. We proceed as in the second part of the proof of Lemma \ref{lemma:fixedpoint-in-omega(c)->maximality}. As $n\geq 3$, and $f(C_{\delta})\subset C_{n+1}$, $\UnstableLocMan{z_{+}}$ connects $\gamma_{1}$ and $\gamma_{n}$, the latter is the left vertical face of $C_{n+1}$. Therefore, there exists an arc $\tau\subset \UnstableLocMan{z_{+}}$ connecting $\gamma_{2}$ and $\gamma_{3}$, see Figure \ref{figure:h(C)-OP}. Consequently, the curves $h(\tau)\subset \UnstableMan{z_{+}}$ and $\gamma_{1}=\StableLocMan{z_{+}} $ bound a disk $W$ that contains $h(C_2)$. This follows from the order on $\pi_1(h(C_k)\cap \Ryc)$ given by Proposition \ref{proposition:multi-renorm} Note that $\fr W\subset \StableMan{z_{+}}\cup\UnstableMan{z_{+}}$, and so $\omega(z)\subset \omega(W)\subset \cl\UnstableMan{z_{+}}$ by Lemma \ref{lemma:manifold-bounded-region->max}, concluding the proof that $\omega(z)\subset \cl\UnstableMan{z_{+}}$ for $z\in Q$.
		
		We have just shown that $\omega(z)\subset \cl\UnstableMan{z_{+}}$ for $z\in Q$. In particular, we see that for any trapping region $T$ with $\cl\UnstableMan{z_{+}}\subset T\subset Q$, we have $\cl\UnstableMan{z_{+}}=\bigcap_{i\geq 0}f^{i}(T)$. The existence of such a trapping region is ensured by Lemma \ref{lemma:trapping-region}. Therefore, $\cl\UnstableMan{z_{+}}$ is a maximal invariant set in $T$, and so $\cl\UnstableMan{z_{+}}$ is a topological attractor.
		
		Let us additionally assume piecewise uniform hyperbolicity \ref{stepA2-hyp} and high expansion \ref{stepA2-expansion}, and show that $\cl\UnstableMan{z_{+}}$ is an attractor. By Corollary \ref{corollary:stmani-dense-in-R}, map $f|_{\cl\UnstableMan{z_{+}}}$ is mixing, and so also $f|_{\cl\UnstableMan{z_{+}}}$ is transitive. Transitivity implies that the set of points with dense orbits is dense $G_{\delta}$. As $\bigcup_{i\in\Z}f^{i}(\Rc)$ is $F_{\sigma}$, there exists a point $z\notin\bigcup_{i\in\Z}f^{i}(\Rc)$ with a dense orbit. Then a positive lower Lyapunov exponent at $(x,(0,1))$ is implied by \ref{stepA2-hyp} via perturbation. This finishes the proof that $\cl\UnstableMan{z_{+}}$ is a topological strange attractor in $K_{n}\times(0,k_{n})$.
		
		The continuity of $\cl\UnstableMan{z_{+}}$ in the Hausdorff metric follows from Proposition \ref{prop:continuity-hausdorff} and the fact that the trapping region $T$ from Lemma \ref{lemma:trapping-region} can be chosen to vary continuously in the Hausdorff metric. Indeed, for $k:=\min\{i\in\N\colon f(Q)\cap C_{i}=\emptyset \}$ we may let $T=S(\beta_{k+1},\gamma_k,Q)$. The continuity of $T$ follows from the fact that $\beta_{k+1}$ and $\gamma_{k}$ are continuous in the Hausdorff metric.
	\end{proof}
    \subsection*{Acknowledgements}
    This article has been created as part of my doctoral thesis. I am grateful to professor Jan Boroński for supervising my doctoral studies and many fruitful discussions on parameterized families with hyperbolicity. I also would like to thank professor Sonja \v{S}timac and Magda Foryś-Krawiec for helpful remarks and discussions.
\bibliography{lozi_map_bib}

@article{ures_1995, 
    AUTHOR = {Ures, Ra\'ul},
     TITLE = {On the approximation of {H}\'enon-like attractors by
              homoclinic tangencies},
   JOURNAL = {Ergodic Theory Dynam. Systems},
  FJOURNAL = {Ergodic Theory and Dynamical Systems},
    VOLUME = {15},
      YEAR = {1995},
    NUMBER = {6},
     PAGES = {1223--1229},
      ISSN = {0143-3857,1469-4417},
   MRCLASS = {58F12 (58F13)},
  MRNUMBER = {1366318},
MRREVIEWER = {Marcelo\ Viana},
       DOI = {10.1017/S0143385700009895},
       URL = {https://doi.org/10.1017/S0143385700009895},
}

@article{dyi-shing-boronski:disc-lozi-att,
	author = {Ou, Dyi-Shing and Boro\'nski, Jan},
	title = {The discontinuity of the Lozi attractor and its prime end
rotation number},
	year = {in preparation}
}

@article{yutaka-ishii1997:kneading1,
	author = {Ishii, Yutaka},
 title = {Towards a kneading theory for {Lozi} mappings. {I}: {A} solution of the pruning front conjecture and the first tangency problem},
 fjournal = {Nonlinearity},
 journal = {Nonlinearity},
 issn = {0951-7715},
 volume = {10},
 number = {3},
 pages = {731--747},
 year = {1997},
 language = {English},
 doi = {10.1088/0951-7715/10/3/008},
 zbMATH = {1190805},
 Zbl = {0908.58015}
}

@book{katok_Hasselblatt_1995, 
	place={Cambridge}, 
	series={Encyclopedia of Mathematics and its Applications}, 
	title={Introduction to the Modern Theory of Dynamical Systems}, 
	publisher={Cambridge University Press}, 
	author={Katok, Anatole and Hasselblatt, Boris}, 
	year={1995}, 
	collection={Encyclopedia of Mathematics and its Applications}
}

@article{simpson-gosh:robust-chaos,
	 AUTHOR = {Ghosh, I. and Simpson, D. J. W.},
     TITLE = {Robust {D}evaney chaos in the two-dimensional border-collision normal form},
   JOURNAL = {Chaos},
  FJOURNAL = {Chaos. An Interdisciplinary Journal of Nonlinear Science},
    VOLUME = {32},
      YEAR = {2022},
    NUMBER = {4},
     PAGES = {Paper No. 043120, 9},
      ISSN = {1054-1500,1089-7682},
   MRCLASS = {37G05 (37G35)},
  MRNUMBER = {4406782},
       DOI = {10.1063/5.0079807},
       URL = {https://doi.org/10.1063/5.0079807},
}

@article{kucharski:str-att,
	 AUTHOR = {Kucharski, Przemys{\l}aw},
     TITLE = {Strange attractors for the family of orientation preserving
              {L}ozi maps},
   JOURNAL = {Chaos},
  FJOURNAL = {Chaos. An Interdisciplinary Journal of Nonlinear Science},
    VOLUME = {33},
      YEAR = {2023},
    NUMBER = {11},
     PAGES = {Paper No. 113121, 13},
      ISSN = {1054-1500,1089-7682},
   MRCLASS = {37C70 (37D45)},
  MRNUMBER = {4667156},
MRREVIEWER = {Yongluo\ Cao},
       DOI = {10.1063/5.0139893},
}

@Article{mora1993,
	AUTHOR = {Mora, Leonardo and Viana, Marcelo},
     TITLE = {Abundance of strange attractors},
   JOURNAL = {Acta Math.},
  FJOURNAL = {Acta Mathematica},
    VOLUME = {171},
      YEAR = {1993},
    NUMBER = {1},
     PAGES = {1--71},
      ISSN = {0001-5962,1871-2509},
   MRCLASS = {58F13 (58F12)},
  MRNUMBER = {1237897},
MRREVIEWER = {Alexander\ Loskutov},
       DOI = {10.1007/BF02392766},
       URL = {https://doi.org/10.1007/BF02392766},
}

@article{banks-brooks-devaney-chaos-1992,
	AUTHOR = {Banks, J. and Brooks, J. and Cairns, G. and Davis, G. and Stacey, P.},
     TITLE = {On {D}evaney's definition of chaos},
   JOURNAL = {Amer. Math. Monthly},
  FJOURNAL = {American Mathematical Monthly},
    VOLUME = {99},
      YEAR = {1992},
    NUMBER = {4},
     PAGES = {332--334},
      ISSN = {0002-9890,1930-0972},
   MRCLASS = {54H20 (58F13)},
  MRNUMBER = {1157223},
MRREVIEWER = {J.\ E.\ Keesling},
       DOI = {10.2307/2324899},
       URL = {https://doi.org/10.2307/2324899},
}

@InProceedings{viana:global-att,
	author="Viana, Marcelo",
	editor="Broer, H. W.
	and van Gils, S. A.
	and Hoveijn, I.
	and Takens, F.",
	title="Global attractors and bifurcations",
	booktitle="Nonlinear Dynamical Systems and Chaos",
	year="1996",
	publisher="Birkh{\"a}user Basel",
	address="Basel",
	pages="299--324",
	isbn="978-3-0348-7518-9"
}

@article{viana-marcelo-strange-att-1993,
	AUTHOR = {Viana, Marcelo},
	TITLE = {Strange attractors in higher dimensions},
	JOURNAL = {Bol. Soc. Brasil. Mat. (N.S.)},
	FJOURNAL = {Boletim da Sociedade Brasileira de Matem\'{a}tica. Nova
	S\'{e}rie},
	VOLUME = {24},
	YEAR = {1993},
	NUMBER = {1},
	PAGES = {13--62},
	ISSN = {0100-3569},
	MRCLASS = {58F13 (58F12)},
	MRNUMBER = {1224299},
	MRREVIEWER = {C.\ Eugene\ Wayne},
	DOI = {10.1007/BF01231695},
	URL = {https://doi.org/10.1007/BF01231695},
}

@article{cao-mao:non-wandering-set-of-some-Henon-maps,
	AUTHOR = {Cao, Yongluo and Mao, Jian Min},
     TITLE = {The non-wandering set of some {H}\'enon maps},
   JOURNAL = {Chaos Solitons Fractals},
  FJOURNAL = {Chaos, Solitons \& Fractals},
    VOLUME = {11},
      YEAR = {2000},
    NUMBER = {13},
     PAGES = {2045--2053},
      ISSN = {0960-0779,1873-2887},
   MRCLASS = {37C70 (37D45 37E99)},
  MRNUMBER = {1771591},
MRREVIEWER = {Ant\'onio\ Pumari\~no},
       DOI = {10.1016/S0960-0779(99)00098-3},
       URL = {https://doi.org/10.1016/S0960-0779(99)00098-3},
}

@misc{dyi-shing-ou:critical-points-I,
	doi = {10.48550/ARXIV.2203.02326},
	
	url = {https://arxiv.org/abs/2203.02326},
	
	author = {Ou, Dyi-Shing},
	
	title = {{Critical points in higher dimensions, I: Reverse order of periodic orbit creations in the Lozi family}},
	
	publisher = {arXiv},
	
	year = {2022},
	
	copyright = {arXiv.org perpetual, non-exclusive license}
}

@article {simpson:robust-chaos,
	AUTHOR = {Glendinning, Paul A. and Simpson, David J. W.},
	TITLE = {A constructive approach to robust chaos using invariant
	manifolds and expanding cones},
	JOURNAL = {Discrete Contin. Dyn. Syst.},
	FJOURNAL = {Discrete and Continuous Dynamical Systems. Series A},
	VOLUME = {41},
	YEAR = {2021},
	NUMBER = {7},
	PAGES = {3367--3387},
	ISSN = {1078-0947},
	MRCLASS = {37G35 (39A28)},
	MRNUMBER = {4247144},
	MRREVIEWER = {Bernhard Lani-Wayda},
	DOI = {10.3934/dcds.2020409},
	URL = {https://doi.org/10.3934/dcds.2020409},
}

@incollection{barge-cont,
	author = {Barge, Marcy},
 title = {Prime end rotation numbers associated with the {H{\'e}non} maps},
 booktitle = {Continuum theory and dynamical systems. Papers of the conference/workshop on continuum theory and dynamical systems held at Lafayette, LA (USA)},
 isbn = {0-8247-9072-3; 978-1-138-43033-4; 978-1-4822-9346-3},
 pages = {15--33},
 year = {1993},
 publisher = {New York: Marcel Dekker, Inc.},
 language = {English},
 zbMATH = {475100},
 Zbl = {0789.58049}
}

@article{glendinning2019robust,
	AUTHOR = {Glendinning, Paul A. and Simpson, David J. W.},
	TITLE = {Robust chaos and the continuity of attractors},
	JOURNAL = {Trans. Math. Appl.},
	FJOURNAL = {Transactions of Mathematics and its Applications. A Journal of
	the IMA},
	VOLUME = {4},
	YEAR = {2020},
	NUMBER = {1},
	PAGES = {tnaa002, 15},
	MRCLASS = {37C70 (37D45 37E05 37G10)},
	MRNUMBER = {4164067},
	MRREVIEWER = {Yongluo Cao},
	DOI = {10.1093/imatrm/tnaa002}
}

@article{henon1976,
	author = {H{\'e}non, M.},
 title = {A two-dimensional mapping with a strange attractor},
 fjournal = {Communications in Mathematical Physics},
 journal = {Commun. Math. Phys.},
 issn = {0010-3616},
 volume = {50},
 pages = {69--76},
 year = {1976},
 language = {English},
 doi = {10.1007/BF01608556},
 zbMATH = {3921578},
 Zbl = {0576.58018}
}

@article{strange-attractor-mis,
	author = {Misiurewicz, Micha{\l}},
	title = {{Strange attractors for the Lozi mappings}},
	journal = {Annals of the New York Academy of Sciences},
	volume = {357},
	number = {1},
	pages = {348-358},
	year = {1980},
	doi = {https://doi.org/10.1111/j.1749-6632.1980.tb29702.x}
}

@article{Young1985BowenRuelleMF,
	author = {Young, Lai-Sang},
 title = {Bowen-{Ruelle} measures for certain piecewise hyperbolic maps},
 fjournal = {Transactions of the American Mathematical Society},
 journal = {Trans. Am. Math. Soc.},
 issn = {0002-9947},
 volume = {287},
 pages = {41--48},
 year = {1985},
 language = {English},
 doi = {10.2307/2000396},
 zbMATH = {3879770},
 Zbl = {0552.58022}
}

@article{henon-dynamics-of,
	 AUTHOR = {Benedicks, Michael and Carleson, Lennart},
     TITLE = {The dynamics of the {H}\'enon map},
   JOURNAL = {Ann. of Math. (2)},
  FJOURNAL = {Annals of Mathematics. Second Series},
    VOLUME = {133},
      YEAR = {1991},
    NUMBER = {1},
     PAGES = {73--169},
      ISSN = {0003-486X,1939-8980},
   MRCLASS = {58F12 (58F13)},
  MRNUMBER = {1087346},
MRREVIEWER = {Feliks\ Przytycki},
       DOI = {10.2307/2944326},
       URL = {https://doi.org/10.2307/2944326},
}

@article{StrangeAttractorswithOneDirectionofInstability,
	AUTHOR = {Wang, Qiudong and Young, Lai-Sang},
     TITLE = {Strange attractors with one direction of instability},
   JOURNAL = {Comm. Math. Phys.},
  FJOURNAL = {Communications in Mathematical Physics},
    VOLUME = {218},
      YEAR = {2001},
    NUMBER = {1},
     PAGES = {1--97},
      ISSN = {0010-3616,1432-0916},
   MRCLASS = {37D45 (37E99)},
  MRNUMBER = {1824198},
MRREVIEWER = {Boris\ Hasselblatt},
       DOI = {10.1007/s002200100379},
       URL = {https://doi.org/10.1007/s002200100379},
}

@book{palistakens:hyp,
	author = {J. Palis, F. Taken},
	title = {Hyperbolicity \& Sensitive Chaotic Dynamics at Homoclinic Bifurcations},
	journal = {SIAM Review},
	volume = {38},
	number = {2},
	pages = {348-349},
	year = {1996},
	doi = {10.1137/1038064},
	eprint = { https://doi.org/10.1137/1038064}
}

@article{Lozi-likemaps,
    AUTHOR = {{Micha{\l} Misiurewicz  and Sonja \v{S}timac}},
	TITLE = {Lozi-like maps},
	JOURNAL = {Discrete Contin. Dyn. Syst.},
	FJOURNAL = {Discrete and Continuous Dynamical Systems. Series A},
	VOLUME = {38},
	YEAR = {2018},
	NUMBER = {6},
	PAGES = {2965--2985},
	ISSN = {1078-0947},
	MRCLASS = {37E30 (37B10 37D45)},
	MRNUMBER = {3809069},
	MRREVIEWER = {Yutaka Ishii},
	DOI = {10.3934/dcds.2018127},
	URL = {https://doi.org/10.3934/dcds.2018127},
}

@article{cao-liu-98,
	AUTHOR = {Cao, Yongluo and Liu, Zengrong},
     TITLE = {Strange attractors in the orientation-preserving {L}ozi map},
   JOURNAL = {Chaos Solitons Fractals},
  FJOURNAL = {Chaos, Solitons \& Fractals},
    VOLUME = {9},
      YEAR = {1998},
    NUMBER = {11},
     PAGES = {1857--1863},
      ISSN = {0960-0779,1873-2887},
   MRCLASS = {58F13},
  MRNUMBER = {1670359},
       DOI = {10.1016/S0960-0779(97)00180-X},
       URL = {https://doi.org/10.1016/S0960-0779(97)00180-X},
}
\bibliographystyle{plain}
    \end{document}